\documentclass [leqno,11pt]{amsart}
\usepackage{amssymb,amsmath, amsthm, color}
\usepackage{cite}
\numberwithin{equation}{section}

\allowdisplaybreaks[3]
\let\f=\frac

\let\om=\omega

\let\Om=\Omega

\let\na=\nabla

\let\pa=\partial

\def\div{\mbox{div}}

\def\curl{\mathop{\rm curl}\nolimits}

\def\bbT{\mathbb{T}}

\newcommand{\newcom}{\newcommand}
\newcommand{\vc}[1]{{\bf #1}}
\newcom{\BT}{{\mathbb{T}^2}}
\newcom{\ve}{\vc{e}}
\newcom{\vN}{\vc{N}}
\newcom{\vn}{\vc{n}}
\newcom{\vG}{\vc{G}}
\newcom{\vF}{\vc{F}}
\newcom{\vf}{\vc{f}}
\newcom{\vg}{\vc{g}}
\newcom{\vq}{\vc{q}}
\newcom{\vu}{\vc{u}}
\newcom{\vv}{\vc{v}}
\newcom{\vV}{\vc{V}}
\newcom{\vw}{\vc{w}}
\newcom{\vW}{\vc{W}}
\newcom{\vb}{\vc{b}}
\newcom{\vh}{\vc{h}}
\newcom{\vz}{\vc{z}}
\newcom{\vup}{\vu^{+}}
\newcom{\vum}{\vu^{-}}
\newcom{\vvp}{\vv^{+}}
\newcom{\vvm}{\vv^{-}}
\newcom{\vbp}{\vb^{+}}
\newcom{\vbm}{\vb^{-}}
\newcom{\vhp}{\vh^{+}}
\newcom{\vhm}{\vh^{-}}
\newcom{\Omp}{{\Om^+}}
\newcom{\Omm}{{\Om^-}}
\newcom{\vupm}{{\vu^{\pm}}}
\newcom{\vvpm}{{\vv^{\pm}}}
\newcom{\vbpm}{{\vb^{\pm}}}
\newcom{\vhpm}{{\vh^{\pm}}}
\newcom{\vFpm}{{\vF^{\pm}}}
\newcom{\vwp}{{\vc{w}^+}}
\newcom{\vwm}{{\vc{w}^-}}
\newcom{\vwpm}{{\vc{w}^{\pm}}}
\newcom{\Ompm}{{\Omega^{\pm}}}
\newcom{\vom}{\boldsymbol{\omega}}
\newcom{\vxi}{\boldsymbol{\xi}}
\newcom{\vop}{\vom^{+}}
\newcom{\vnu}{\boldsymbol{\nu}}
\newcom{\vopm}{\vom^{\pm}}
\newcom{\vjp}{\vj^+}
\newcom{\vjm}{\vj^-}
\newcom{\vjpm}{\vj^{\pm}}
\newcom{\vj}{\boldsymbol{\xi}}
\newcommand{\beq}{\begin{equation}}
\newcommand{\eeq}{\end{equation}}
\newcommand{\ben}{\begin{eqnarray}}
\newcommand{\een}{\end{eqnarray}}
\newcommand{\beno}{\begin{eqnarray*}}
\newcommand{\eeno}{\end{eqnarray*}}
\def\eqdefa{\buildrel\hbox{\footnotesize def}\over =}

\newtheorem{theorem}{Theorem}[section]
\newtheorem{definition}[theorem]{Definition}
\newtheorem{lemma}[theorem]{Lemma}
\newtheorem{proposition}[theorem]{Proposition}

\newtheorem{remark}[theorem]{Remark}

\begin{document}
\title[Free boundary problem in incompressible MHD]{Well-posedness of the free boundary problem in incompressible MHD with surface tension}

\author{Changyan Li}
\address{School of  Mathematical Sciences, Peking University, Beijing 100871, China}
\email{lcy941024@pku.edu.cn}

\author{Hui Li}
\address{Department of Mathematics, Zhejiang University, Hangzhou 310027, China}
\email{lihui92@zju.edu.cn}

\begin{abstract}
In this paper, we study the two phase flow problem with surface tension in the ideal incompressible magnetohydrodynamics. We first prove the local well-posedness of the two phase flow problem with surface tension, then demonstrate that as surface tension tends to zero, the solution of the two phase flow problem with surface tension converges to the solution of the two phase flow problem without surface tension.
\end{abstract}

\maketitle

%\tableofcontents

\section{Introduction}
\subsection{Presentation of the problem}
In this paper, we consider the two phase flow problem with surface tension in the ideal incompressible MHD. The incompressible MHD system can be written as
\begin{equation}\label{eq-ori}
\left\{\begin{aligned}
&\rho\partial_t \textbf{u}+\rho\textbf{u}\cdot\nabla\textbf{u}-\textbf{h}\cdot\nabla\textbf{h}
+\nabla p=0\quad\quad &\mathrm{in} \quad \mathcal{Q}_T,\\
&\mathrm{div} \textbf{u}=0, \quad \mathrm{div} \textbf{h}=0\quad\quad &\mathrm{in} \quad \mathcal{Q}_T,\\
&\partial_t \textbf{h}+\textbf{u}\cdot\nabla\textbf{h}-\textbf{h}\cdot\nabla\textbf{u}=0\quad\quad &\mathrm{in} \quad \mathcal{Q}_T,\\
\end{aligned}
\right.
\end{equation}
where $\textbf{u}$ is the fluids velocity, $\textbf{h}$ is the magnetic field, $p$ denotes the pressure.
We study the solution of \eqref{eq-ori} which are smooth on each side of a smooth interface $\Gamma(t)$ in a domain $\Omega$. More precisely, we let
\begin{align*}
\Omega=\mathbb{T}^2\times[-1,1], \quad \Gamma(t)=\{x\in\Omega|x_3=f(t,x'),x'=(x_1,x_2)\in\mathbb{T}^2\},
\end{align*}
\begin{align*}
\Omega_t^\pm=\{x\in\Omega|x_3\gtrless f(t,x'),x'\in\mathbb{T}^2\},\quad \mathcal{Q}_T^\pm=\bigcup_{t\in(0,T)}\{t\}\times\Omega_t^\pm.
\end{align*}
For simplicity of notation we write $\rho_{\Omega_t^\pm}=\rho^\pm$, where $\rho^\pm$ are two constants that represent the density of the fluids on each side of the free boundary. We also define
\begin{align*}
\textbf{u}^{\pm}:=\textbf{u}|_{\Omega_t^\pm},\quad \textbf{h}^{\pm}:=\textbf{h}|_{\Omega_t^\pm},\quad p^{\pm}:=p|_{\Omega_t^\pm},\quad
\end{align*}
which are smooth in $\mathcal{Q}_T$ and satisfy
\begin{equation}\label{equa2}
\left\{\begin{aligned}
&\rho^{\pm}\partial_t \textbf{u}^{\pm}+\rho^{\pm}\textbf{u}^{\pm}\cdot\nabla\textbf{u}^{\pm}-\textbf{h}\cdot\nabla\textbf{h}^{\pm}
+\nabla p^{\pm}=0 \quad\quad &\mathrm{in} \quad \mathcal{Q}_T^{\pm},\\
&\mathrm{div} \textbf{u}^{\pm}=0, \quad \mathrm{div} \textbf{h}^{\pm}=0\quad\quad &\mathrm{in} \quad \mathcal{Q}_T^{\pm}, \\
&\partial_t \textbf{h}^{\pm}+\textbf{u}^{\pm}\cdot\nabla\textbf{h}^{\pm}-\textbf{h}^{\pm}\cdot\nabla\textbf{u}^{\pm}=0\quad\quad &\mathrm{in} \quad \mathcal{Q}_T^{\pm}.
\end{aligned}
\right.
\end{equation}
On the moving interface $\Gamma_t$, we impose the following boundary conditions:

\begin{align}\label{equa:p}
&[p]:=p^+-p^-=\sigma H(f)=\sigma\nabla_{x'}\cdot
(\frac{\nabla_{x'}f}{\sqrt{1+|\nabla_{x'}f|^2}}),\\
& \textbf{u}^{\pm}\cdot \mathbf{N}=\partial_t f, \quad \textbf{h}^{\pm}\cdot \mathbf{N}=0 \quad \mathrm{on} \quad \Gamma_t,\label{bc-uhf}
\end{align}
where $\sigma$ is the surface tension coefficient, $H(f)$ is the mean curvature of the surface, $\mathbf{N}=(-\partial_1 f,-\partial_2 f,1)$ is the normal vector of the surface. Condition \eqref{equa:p} means that there is surface tension acting on the free boundary. Condition \eqref{bc-uhf} means that the free boundary is moving with the fluid, and the magnetic will not pass through the free boundary.

On the artificial boundary $\Gamma^\pm=\bbT^2\times\{\pm1\}$, we also assume that
\begin{align}\label{bc-uh}
u_3^\pm=0,\quad h_3^\pm=0 \quad \mathrm{on} \quad \Gamma^\pm.
\end{align}

The system (\ref{equa2}) is supplement with the initial data:
\begin{align}
\textbf{u}^{\pm}(0,x)=\textbf{u}^{\pm}_0(x), \quad \textbf{h}^{\pm}(0,x)=\textbf{h}^{\pm}_0(x) \quad \mathrm{in} \quad \Omega_0^\pm,
\end{align}
which satisfies
\begin{equation}\label{con-ini}
\left\{\begin{aligned}
&\mathrm{div}\textbf{u}_0^{\pm}=0,\mathrm{div}\textbf{h}_0^{\pm}=0 \quad & \mathrm{in}\quad \Omega_0^\pm,\\
&\textbf{u}_0^{+}\cdot\textbf{N}_0=\textbf{u}_0^{-}\cdot\textbf{N}_0,\textbf{h}_0^{+}\cdot\textbf{N}_0=\textbf{h}_0^{-}\cdot\textbf{N}_0=0 \quad& \mathrm{on} \quad\Gamma_0,\\
&u_{03}^\pm=0,\ h_{03}^\pm=0\quad& \mathrm{on} \quad\Gamma^\pm.
\end{aligned}
\right.
\end{equation}
The system \eqref{equa2}-\eqref{con-ini} is called the two-phase flow problem for incompressible MHD. One of main goals in this paper is to study the local well-posedness and the zero surface tension limit of this system.

We remark that the divergence-free restriction on $\vh^\pm$ is a compatibility condition. Applying the divergence operator to the third equation of \eqref{equa2}, we have
\begin{align*}
	\pa_t\div\vh^\pm+\vu^\pm\cdot\nabla\div\vh^\pm=0.
\end{align*}
Therefore, if $\div\vh^\pm_0=0$, the solution of \eqref{equa2}-\eqref{equa:p} will satisfies $\div\vh^\pm=0$ for $\forall t>0$. A similar argument can be applied to yield that $\vh^\pm\cdot\vN=0$ if $\vh^\pm_0\cdot\vN_{0}=0$.
\subsection{Background and related works}
In inviscid flow, a surface across which there is a discontinuity in fluid velocity is called a vortex sheet. In the absence of surface tension and magnetic field, it is well known that the vortex sheet problem of incompressible fluids is ill-posed due to the Kelvin-Helmholtz instability \cite{MajB}. During the past several decades, researches have found that such instability can be stabilized by surface tension. For irrotational flow, Ambrose \cite{AM1} and Ambrose-Masmoudi \cite{AM} proved the local well-posedness of vortex sheets with surface tension for in two and three dimensions respectively. For general problem with vorticity, Shatah-Zeng \cite{SZ} established a priori estimates in a geometric approach, and Cheng-Coutand-Shkoller \cite{CCS,CCS1} proved the local well-posedness of the three dimensional problems. For other results about the vortex sheet problems, we refer the readers to \cite{BHL,CO,Wsj3}.

In the mid-twentieth century, Syrovatskij \cite{Syr} and Axford \cite{Axf} found that the magnetic field has a stabilization effect on the Kelvin-Helmholtz instability. The Syrovatskij stability condition can be expressed as:
\begin{align}\label{2}
|[\mathbf{u}]|^2\le 2(|\mathbf{h}^+|^2+|\mathbf{h}^-|^2),\text{ on }\Gamma_t,
\end{align}
\begin{align*}
|[\mathbf{u}]\times\mathbf{h}^+|^2+|[\mathbf{u}]\times\mathbf{h}^-|^2\le 2|\mathbf{h}^+\times\mathbf{h}^-|^2,\text{ on }\Gamma_t,
\end{align*}
where $[\mathbf{u}]=(\vu^+-\vu^-)$.

In the recent decades, great progress has been made in studying the stabilizing effect of the Syrovatskij condition (\ref{2}). Morando-Trakhinin-Trebeschi\cite{MTT} proved a priori estimates with a loss of derivatives for the linearized system. Furthermore,
 under a strong stability condition
 \begin{align}\label{4}
\max(|[\mathbf{u}]\times\mathbf{h}^+,[\mathbf{u}]\times\mathbf{h}^-|)<|\mathbf{h}^+\times\mathbf{h}^-|,\text{ on }\Gamma_t,
\end{align}
Trakhinin\cite{Tra} proved an a priori estimate for the linearized problem without loss of derivative. For the nonlinear current-vortex sheet problem, Coulombel-Morando-Secchi-Trebeschi\cite{CMS} proved an a priori estimate under the strong stability condition (\ref{4}).  Recently, Sun-Wang-Zhang\cite{SWZ1} gave the first rigorous confirmation of the stabilizing effect of the magnetic field on Kelvin-Helmholtz instability under the Syrovatskij stability condition (\ref{4}). We also refer to some related works \cite{Tra1, Tra2, Chen, WY} on the compressible problem and works \cite{Tra-JDE, ST,HL, Hao, GW} on the plasma-vacuum problem.

The aim of this paper is to show the local well-posednss for the current-vortex sheet problem with surface tension. That is to say, the magnetic field do not destroy the stabilization effect of surface tension. Under additional assumption that the Syrovatskij condition holds, we also show that, as surface tension tends to zero, the solution of the two phase flow problem with surface tension converges to the solution of the two phase flow problem without surface tension. The framework we used in this paper is developed in \cite{SWZ1}. The basic idea is study the evolution equation of the free surface where the surface tension leads to a third-order term. Inspired by Alazard-Burq-Zuliy \cite{ABZ1}, we use paraproduct decomposition to analysis the most nontrivial third-order term, and find that the evolution equation is strictly hyperbolic.

In the free boundary problem of inviscid flow where there is only one fluid, the Rayleigh-Taylor instability, instead of the Kelvin-Helmholtz instability, may occur. There are a lot of remarkable literatures studying such problems \cite{Wsj1, Wsj2, ZpZzf, SZ2,CS,Da,ABZ,CL}. From a mathematical point of view, the elastodynamics have similar structures to the magnetohydrodynamics. In a very recent work, Gu-Lei \cite{GL} proved the local well-posedness of the free-boundary in incompressible elastodynamics with surface tension.
\subsection{Main results}
Now, let us state our main results.
\begin{theorem}\label{thm:1}
Assume $s\ge 6$ is an integer and $f_0\in H^{s+1}(\mathbb{T}^2)$,$\textbf{u}^{\pm}_0, \textbf{h}^{\pm}_0\in H^{s}(\Omega_0^\pm)$, $\sigma>0$, $\rho^+=\rho^-=1$, moreover we assume that there exists $c_0\in(0,\frac{1}{2})$ so that
\begin{align*}
-(1-2c_0)\le f_0\le(1-2c_0).
\end{align*}
Then there exists a time $T>0$ such that system \eqref{equa2}-\eqref{con-ini} admits a unique solution $(f,\textbf{u},\textbf{h})$ in $[0,T]$ satisfying
\begin{itemize}
\item[1.]$f\in L^\infty([0,T),H^{s+1}(\mathbb{T}^2))$,
\item[2.]$\textbf{u}^\pm,\textbf{h}^\pm\in L^\infty([0,T),H^{s}(\Omega^\pm_t))$,
\item[3.]$-(1-c_0)\le f\le (1-c_0)$.
\end{itemize}
\end{theorem}

Before state the result of zero surface tension limit, we introduce a Syrovatskij type stability condition:
\begin{align}\label{con-sta}
\Lambda(\textbf{h}^{\pm},[\textbf{u}])\eqdefa&\inf_{x\in\Gamma_t}\inf_{\varphi_1^2+\varphi_2^2=1}\frac{1}{\rho^+ +\rho^-}(\underline{h_1^+}\varphi_1+\underline{h_2^+}\varphi_2)^2+
\frac{1}{\rho^+ +\rho^-}(\underline{h_1^-}\varphi_1+\underline{h_2^-}\varphi_2)^2\\
&-(v_1\varphi_1+v_2\varphi_2)^2\ge c_0>0,\nonumber
\end{align}
where $v_i=\frac{\sqrt{\rho^+\rho^-}}{\rho^++\rho^-}[u_i]$.

With such stability condition, Sun-Wang-Zhang\cite{SWZ1} prove the local well-posedness of current-vortex sheet problem without surface tension for the case $\rho^+=\rho^-=1$ and we \cite{LL} get the similar results for the general case $\rho^+,\rho^->0$.

Under the assumption that the initial data satisfies the stability condition \eqref{con-sta}, we prove that as $\sigma$ tends to 0, the solution of the two-phase flow problem got in \cite{SWZ1} is the limit of the solutions got in Theorem \ref{thm:1}. Indeed, we have the following result.
\begin{theorem}\label{theo2}
Assume $s\ge 6$ is an integer and $f_0\in H^{s+1}(\mathbb{T}^2)$,$\textbf{u}^{\pm}_0, \textbf{h}^{\pm}_0\in H^{s}(\Omega_0^\pm)$, $\sigma>0$, $\rho^+=\rho^-=1$, moreover we assume that there exists $c_0\in(0,\frac{1}{2})$ so that
\begin{itemize}
	\item[1.] $-(1-2c_0)\le f_0\le(1-2c_0)$,
	\item[2.] $\Lambda(\textbf{h}_0^{\pm},[\textbf{u}_0])\ge 2c_0$.
\end{itemize}
Then there exist $T>0$ independent of $\sigma$ such that system \eqref{equa2}-\eqref{con-ini} admits a unique solution $(f^\sigma,\textbf{u}^\sigma,\textbf{h}^\sigma)$ in $[0,T]$ satisfying
\begin{itemize}
\item[1.]$f^\sigma\in L^\infty([0,T),H^{s+1}(\mathbb{T}^2))$,
\item[2.]$\textbf{u}^{\sigma\pm},\textbf{h}^{\sigma\pm}\in L^\infty([0,T),H^{s}(\Omega^\pm_t))$,
\item[3.]$-(1-c_0)\le f^\sigma\le (1-c_0)$,
\item[4.]$\Lambda(\textbf{h}^{\pm},[\textbf{u}])\ge c_0$.
\end{itemize}
Moreover, as $\sigma$ tends to 0, the solution $(f^\sigma,\textbf{u}^\sigma,\textbf{h}^\sigma)$ converges to the solution $(f,\textbf{u},\textbf{h})$ of the system \eqref{equa2}-\eqref{con-ini} with $\sigma=0$.
\end{theorem}
\begin{remark}
Our method is also applicable to the general case $\rho^+,\rho^->0$. In this case, the surface tension term is a little more complex, however the evolution equation of the free surface is also strictly hyperbolic. As surface tension goes to 0, the limit of solutions to this problem is the solution got in \cite{LL}. For the one fluid problem that there is no fluid and no magnetic in the upper domain, we can also prove local well-posedness by using the method developed herein. The key steps to prove these results can be found in Section 7.
\end{remark}

The rest of this paper is organized as follows. In Section 2, we will introduce the reference domain, harmonic coordinate, and the Dirichlet-Neumann operator. In Section 3, we reformulate the system into a new formulation. Section 4 provides the uniform estimates for the linearized system. In Section 5 and Section 6, we construct an iteration map and prove the existence and uniqueness of the solution. Section 7 shows that the approach developed in this paper can be applied to some other cases.
\section{Reference domain, harmonic coordinate and Dirichlet-Neumann Operator}
In this section, we recall some fundamental lemmas on the harmonic coordinate and Dirichlet-Neumann operators.

We first introduce some notations used throughout this paper. We denote by $C(\cdot,\cdot)$ a positive constant or a positive nondecreasing function depending only on its variables which may be different from line to line. We use $x=(x_1,x_2,x_3)$ to denote the coordinates in the fluid region, and use $x'=(x_1,x_2)$ to denote the natural coordinates on the interface or on the top/bottom boundary $\Gamma^\pm$. In addition, we will use the Einstein summation notation where a summation from 1 to 2 is implied over repeated index (i.e. $a_ib_i = a_1b_1 + a_2b_2$).

For a function $g:\Omega_f^\pm\to\mathbb R$, we denote $\nabla g=(\pa_1g,\pa_2g,\pa_3g)$, and for a function $\eta:\bbT^2\to \mathbb R$, $\nabla\eta=(\pa_1\eta,\pa_2\eta)$. For a function $g:\Omega_f^\pm\to\mathbb R$, we can define its trace on $\Gamma_f$, which are denoted by $\underline g(x')$. Thus, for $i=1,2$,
\begin{align*}
  \pa_i\underline g(x')=\pa_i g(x',f(x'))+\pa_3g(x',f(x'))\pa_if(x').
\end{align*}
We denote by $||\cdot||_{H^s(\Omega_f^\pm)}$, $||\cdot||_{H^s}$ the Sobolev norm on $\Omega_f^\pm$ and $\bbT^2$ respectively. Moreover, for operator $P$ defined on $H^s(\bbT^2)$, we denote its operator norm by
\begin{align*}
	\|P\|_{H^s\rightarrow H^k}=\sup_{||f||_{H^s}\le1}||Pf||_{H^k}.
\end{align*}

To solve the free boundary problem, we introduce a fixed reference domain. Let $\Gamma_*$ be a fixed graph given by
\begin{equation*}
\Gamma_*=\{(y_1,y_2,y_3):y_3=f_*(y_1,y_2)\},
\end{equation*}
where $f_*$ satisfies $\int_{\bbT^2}f_*(y')dy'=0$. The reference domain is given by
\begin{equation*}
\Omega_*=\mathbb{T}^2\times(-1,1), \quad \Omega_*^\pm=\{y\in\Omega_*|y_3\gtrless f_*(y_1,y_2),y'\in\bbT^2\}.
\end{equation*}
We will look for a free boundary that lies close to the reference domain. For this purpose, we define
\begin{equation*}
\Upsilon(\delta,k):=\{f\in H^k(\mathbb{T}^2):\|f-f_*\|_{H^k(\mathbb{T}^2)}\le\delta\}.
\end{equation*}
For $f\in\Upsilon(\delta,k)$, we define $\Gamma_f, \Omega_f^+, \Omega_f^-$ by
\begin{equation*}
\Gamma_f:=\{x\in\Omega_t|x_3=f(t,x'),x'\in\bbT^2\},\quad\Omega_f^\pm=\{x\in\Omega_t|x_3\gtrless f(t,x'),x'\in\bbT^2\}.
\end{equation*}
We denote by $\mathbf{N}_f:=(-\partial_1 f,-\partial_2 f,1)$ the outward normal vector of $\Omega_f^-$ on $\Gamma_f$, and $\mathbf{n}_f:=\mathbf{N}_f/\sqrt{1+|\nabla f|^2}$.
Then we need to introduce the harmonic coordinate. For given $f\in\Upsilon(\delta,k)$, we define a map $\Phi_f^\pm:\Omega_*^\pm\rightarrow\Omega_f^\pm$ by the harmonic extension:
\begin{equation}
\left\{\begin{aligned}
&\Delta_y\Phi_f^\pm=0 \quad \quad &y\in \Omega_*^\pm,\\
&\Phi_f^\pm(y',f_*(y'))=(y',f(y')) \quad \quad &y'\in \mathbb{T}^2,\\
&\Phi_f^\pm(y',\pm 1)=(y',\pm 1) \quad \quad &y'\in \mathbb{T}^2.
\end{aligned}
\right.
\end{equation}
For each $\Gamma_*$, there exists $\delta_0=\delta_0(\|f_*\|_{W^{1,\infty}})>0$ so that $\Phi_f^\pm$ is a bijection whenever $\delta\le\delta_0$. Then, there exists an inverse map $\Phi_f^{\pm -1}:\Omega_f^\pm\rightarrow\Omega_*^\pm$ such that
\begin{equation*}
\Phi_f^{\pm -1}\circ\Phi_f^\pm=\Phi_f^\pm\circ\Phi_f^{\pm -1}=Id.
\end{equation*}
We list some properties of the harmonic coordinate (see \cite{SWZ1} for example):
\begin{lemma}
Let $f\in\Upsilon(\delta_0,s-\frac{1}{2})$ for $s\ge 3$. Then there exists a constant C depending only on $\delta_0$ and $\|f_*\|_{H^{s-\frac{1}{2}}}$ so that\\
1. If $u\in H^\sigma(\Omega^\pm_f)$ for $\sigma\in[0,s]$, then
\begin{equation*}
\|u\circ\Phi_f^\pm\|_{H^\sigma(\Omega_*^\pm)}\le C\|u\|_{H^\sigma(\Omega_f^\pm)}.
\end{equation*}
2. If $u\in H^\sigma(\Omega^\pm_*)$ for $\sigma\in[0,s]$, then
\begin{equation*}
\|u\circ\Phi_f^{\pm -1}\|_{H^\sigma(\Omega_f^{\pm })}\le C\|u\|_{H^\sigma(\Omega_*^\pm)}.
\end{equation*}
3. If $u,v\in H^\sigma(\Omega^\pm_*)$ for $\sigma\in[2,s]$, then
\begin{equation*}
\|uv\|_{H^\sigma(\Omega_f^{\pm})}\le C\|u\|_{H^\sigma(\Omega_f^\pm)}\|v\|_{H^\sigma(\Omega_f^\pm)}.
\end{equation*}
\end{lemma}
Now we introduce the Dirichlet-Neumann operator which maps the Dirichlet boundary value of a harmonic function to its Neumann boundary value. For any $g(x')\in H^k(\bbT^2)$, we denote by $\mathcal{H}^\pm_f g$ the harmonic extension from $\Gamma_f$ to $\Omega_f^\pm$:
\begin{equation}
\left\{\begin{aligned}
&\Delta\mathcal{H}^\pm_f g=0 \quad \quad &x\in \Omega_*^\pm,\\
&(\mathcal{H}^\pm_f g)(x',f(x'))=g(x') \quad \quad &x'\in \mathbb{T}^2,\\
&\partial_3\mathcal{H}^\pm_f g(x',\pm 1)=0 \quad \quad &x'\in \mathbb{T}^2.
\end{aligned}
\right.
\end{equation}
Then we define the Dirichlet-Neumann operator:
\begin{equation*}
\mathcal{N}^\pm_f g \overset{\text{def}}{=}\mp\textbf{N}_f\cdot(\nabla\mathcal{H}^\pm_f g)\mid_{\Gamma_f}.
\end{equation*}
We will use the following properties from \cite{ABZ,SWZ1}.
\begin{lemma}
It holds that\\
1. $\mathcal{N}^\pm_f$ is a self-adjoint operator:
\begin{equation*}
(\mathcal{N}^\pm_f\psi,\phi)=(\psi,\mathcal{N}^\pm_f\phi), \quad \forall \phi,\psi \in H^{\frac{1}{2}}(\mathbb{T}^2);
\end{equation*}
2. $\mathcal{N}^\pm_f$ is a positive operator:
\begin{equation*}
(\mathcal{N}^\pm_f\phi,\phi)=\|\nabla\mathcal{H}^\pm_f\phi\|^2_{L^2(\Omega_f^\pm)}\ge 0, \quad  \forall\phi\in H^{\frac{1}{2}}(\mathbb{T}^2);
\end{equation*}
Especially, if $\int_{\mathbb{T}^2}\phi(x')dx'=0$, there exists $c>0$ depending on $c_0$, $\|f\|_{W^{1,\infty}}$ such that
\begin{equation*}
(\mathcal{N}^\pm_f\phi,\phi)\ge c\|\mathcal{H}^\pm_f\phi\|^2_{H^1(\Omega_f^\pm)}\ge c\|\phi\|_{H^{\frac{1}{2}}}^2, \quad  \forall\phi\in H^{\frac{1}{2}}(\mathbb{T}^2).
\end{equation*}
3. $\mathcal{N}^\pm_f$ is a bijection from $H^{k+1}_0(\mathbb{T}^2)$ to $H^{k}_0(\mathbb{T}^2)$ for $k\ge 0$, where
\begin{equation*}
H^{k}_0(\mathbb{T}^2):=H^{k}(\mathbb{T}^2)\bigcap\{\phi\in L^2(\mathbb{T}^2):\int_{\mathbb{T}^2}\phi(x')dx'=0\}.
\end{equation*}
\end{lemma}
\section{Reformulation of the problem}
In this section, we derive a new system that is equivalent to the original system \eqref{equa2}-\eqref{bc-uh}. The new system consists of the following quantities:
\begin{itemize}
\item The height function of the interface: $f$;
\item The scaled normal velocity on the interface: $\theta=\textbf{u}^{\pm}\cdot  {\bf{N}}_{f}$;
\item The vorticity and current in the fluid region: $\vom=\nabla\times\textbf{u}, \vxi=\nabla\times\textbf{h}$;
\item The average of the tangential part of the velocity and the magnetic field on the top and bottom fixed boundary:
\end{itemize}
\begin{equation*}
\mathfrak{a}_i^\pm(t)=\int_{\mathbb{T}^2}u_i^\pm(t,x',\pm 1)dx',\quad \mathfrak{b}_i^\pm(t)=\int_{\mathbb{T}^2}h_i^\pm(t,x',\pm 1)dx'(i=1,2).
\end{equation*}
\subsection{Evolution of the Scaled Normal Velocity}
Let
\begin{equation}
\theta(t,x')\overset{\text{def}}{=}\textbf{u}^{\pm}(t,x',f(t,x'))\cdot {\bf{N}}_{f}(t,x'),
\end{equation}
we have
\begin{equation}
\partial_t f(t,x')=\theta(t,x').
\end{equation}
In this subsection, we will derive the evolution equation of $\theta$. To this end, we need the following  elementary lemma, which can be proved by direct calculation.
\begin{lemma}\label{lem-cal}\cite{SWZ1}
For $\mathbf{u}=\mathbf{u}^\pm,\mathbf{h}^\pm$, we have
\begin{equation}
(\textbf{u}\cdot\nabla\textbf{u})\cdot{\bf{N}}_{f}-\partial_3 u_j N_j({\bf{u}}\cdot{{\bf{N}}}_{f} )|_{x_3=f(t,x')}=
\underline{u}_1\partial_1(\underline{u}_j N_j)+\underline{u}_2\partial_2(\underline{u}_j N_j)+\sum_{i,j=1,2}\underline{u}_i\underline{u}_j
\partial_i\partial_j f.
\end{equation}
\end{lemma}
With the help of Lemma \ref{lem-cal}, we deduce from \eqref{equa2} that
\begin{align*}
	\partial_t \theta&=(\partial_t {\bf{u}}^{+}+\partial_3 {\bf{u}}^{+}\partial_t f)\cdot{{\bf{N}}}_{f}+ {\bf{u}}^{+}\cdot\partial_t{{\bf{N}}}_f|_{x_3=f(t,x')}\\
=&(-\textbf{u}^{+}\cdot\nabla\textbf{u}^{+}+\frac{1}{\rho^{+}}\textbf{h}^{+}\cdot\nabla\textbf{h}^{+}-\frac{1}{\rho^{+}}\nabla p^+
+\partial_3\textbf{u}^{+}\partial_t f)\cdot\textbf{N}_f-\textbf{u}^{+}\cdot(\partial_1\partial_t f,\partial_2\partial_t f,0)|_{x_3=f(t,x')}\\
=&((-\textbf{u}^{+}\cdot\nabla)\textbf{u}^{+}+\partial_3\textbf{u}^{+}(\textbf{u}^{+}\cdot\textbf{N}_f))\cdot\textbf{N}_f+
\frac{1}{\rho^{+}}(\textbf{h}^{+}\cdot\nabla)\textbf{h}^{+}\cdot\textbf{N}_f\\
&-\frac{1}{\rho^{+}}\textbf{N}_f\cdot\nabla p^+-
\textbf{u}^{+}\cdot(\partial_1 \theta,\partial_2 \theta,0)|_{x_3=f(t,x')}\\
=&-2(\underline{u}_1^+\partial_1 \theta+\underline{u}_2^+\partial_2 \theta)-\frac{1}{\rho^{+}}\textbf{N}_f\cdot\underline{\nabla p}^+
-\sum_{i,j=1,2}\underline{u}_i^+\underline{u}_j^+\partial_i\partial_j f+\frac{1}{\rho^{+}}\sum_{i,j=1,2}\underline{h}_i^+\underline{h}_j^+\partial_i\partial_j f,
\end{align*}
and similarly,
\begin{equation*}
\partial_t \theta=-2(\underline{u}_1^-\partial_1 \theta+\underline{u}_2^-\partial_2 \theta)-\frac{1}{\rho^{-}}\textbf{N}_f\cdot\underline{\nabla p}^-
-\sum_{i,j=1,2}\underline{u}_i^-\underline{u}_j^-\partial_i\partial_j f+\frac{1}{\rho^{-}}\sum_{i,j=1,2}\underline{h}_i^-\underline{h}_j^-\partial_i\partial_j f.
\end{equation*}
Therefore, it holds that
\begin{equation}\label{1}
\begin{aligned}
&2(\underline{u}_1^+\partial_1 \theta+\underline{u}_2^+\partial_2 \theta)+\frac{1}{\rho^{+}}\textbf{N}_f\cdot\underline{\nabla p}^+
+\sum_{i,j=1,2}\underline{u}_i^+\underline{u}_j^+\partial_i\partial_j f-\frac{1}{\rho^{+}}\sum_{i,j=1,2}\underline{h}_i^+\underline{h}_j^+\partial_i\partial_j f\\
=&2(\underline{u}_1^-\partial_1 \theta+\underline{u}_2^-\partial_2 \theta)+\frac{1}{\rho^{-}}\textbf{N}_f\cdot\underline{\nabla p}^-
+\sum_{i,j=1,2}\underline{u}_i^-\underline{u}_j^-\partial_i\partial_j f-\frac{1}{\rho^{-}}\sum_{i,j=1,2}\underline{h}_i^-\underline{h}_j^-\partial_i\partial_j f.
\end{aligned}
\end{equation}

From the first equation of \eqref{equa2} and the boundary condition \eqref{bc-uh}, we get
\begin{equation*}
\Delta p^{\pm}=\mathrm{tr}(\nabla \textbf{h}^{\pm})^2-\rho^\pm \mathrm{tr}(\nabla \textbf{u}^{\pm})^2\quad\text{ in } \Omega_f^\pm,
\end{equation*}
and
\begin{align*}
	\pa_3p^\pm=0\quad\text{ on } \Gamma^\pm.
\end{align*}
Recalling the definition of harmonic extension $\mathcal{H}^\pm_f$, we have the following representation for the pressure $p^{\pm}$:
\begin{equation*}
p^{\pm}=\mathcal{H}^\pm_f\underline{p}^\pm+\rho^\pm p_{\textbf{u}^{\pm},\textbf{u}^{\pm}}-p_{\textbf{h}^{\pm},\textbf{h}^{\pm}},
\end{equation*}
where $p_{\textbf{v}^{\pm},\textbf{v}^{\pm}}$ denotes the solution of the elliptic equation
\begin{equation}\label{eq-p}
\left\{\begin{aligned}
&\Delta p_{\textbf{v}^{\pm},\textbf{v}^{\pm}}=-\mathrm{tr}(\nabla\textbf{v}^{\pm} \nabla\textbf{v}^{\pm})\quad\quad &\mathrm{in} \quad \Omega_f^{\pm},\\
&p_{\textbf{v}^{\pm},\textbf{v}^{\pm}}=0\quad\quad &\mathrm{on} \quad \Gamma_f,\\
&\textbf{e}_3\cdot\nabla p_{\textbf{v}^{\pm},\textbf{v}^{\pm}}=0\quad\quad &\mathrm{on} \quad \Gamma^\pm.
\end{aligned}
\right.
\end{equation}
Thus, we infer from (\ref{1}) that
\begin{align*}
	&\frac{1}{\rho^+}\textbf{N}_f\cdot\nabla\mathcal{H}^+_f\underline{p}^+
-\frac{1}{\rho^-}\textbf{N}_f\cdot\nabla\mathcal{H}^-_f\underline{p}^-\\
=&-\big[2(\underline{u}_1^+\partial_1 \theta+\underline{u}_2^+\partial_2 \theta)+
\textbf{N}_f\cdot\underline{\nabla(p_{\textbf{u}^{+},\textbf{u}^{+}}-\frac{1}{\rho^+}p_{\textbf{h}^{+},\textbf{h}^{+}})}+
\sum_{i,j=1,2}(\underline{u}_i^+\underline{u}_j^+-\frac{1}{\rho^{+}}\underline{h}_i^+\underline{h}_j^+ )\partial_i\partial_j  f\big]\\
&+\big[2(\underline{u}_1^-\partial_1 \theta+\underline{u}_2^-\partial_2 \theta)+
\textbf{N}_f\cdot\underline{\nabla(p_{\textbf{u}^{-},\textbf{u}^{-}}-\frac{1}{\rho^-}p_{\textbf{h}^{-},\textbf{h}^{-}})}+
\sum_{i,j=1,2}(\underline{u}_i^-\underline{u}_j^--\frac{1}{\rho^{-}}\underline{h}_i^-\underline{h}_j^- )\partial_i\partial_j  f\big]\\
\overset{\Delta}{=}&-g^++g^-.
\end{align*}
Recalling the definition of Dirichlet-Neumann operator, we rewrite the above equality as
\begin{equation*}
-\frac{1}{\rho^+}\mathcal{N}^+_f \underline{p}^+-\frac{1}{\rho^-}\mathcal{N}^-_f \underline{p}^-=-g^+ + g^-.
\end{equation*}
As $\underline{p}^+-\underline{p}^-=\sigma H(f)$ on $\Gamma_f$, we have
\begin{align*}
	\underline{p}^\pm=\widetilde{\mathcal{N}_f}^{-1}\big(g^+ - g^-\pm\frac{1}{\rho^\mp}\mathcal{N}_f^\mp \sigma H(f)\big)
\end{align*}
where
\begin{align*}
	\widetilde{\mathcal{N}_f}\overset{\text{def}}{=}\frac{1}{\rho^+}\mathcal{N}^+_f+\frac{1}{\rho^-}\mathcal{N}^-_f.
\end{align*}
Moreover, it's easy to see
\begin{equation*}
\mathcal{N}^+_f=\big(\frac{1}{\rho^+}+\frac{1}{\rho^-})^{-1}(\widetilde{\mathcal{N}_f}+\frac{1}{\rho^-}(\mathcal{N}^+_f-\mathcal{N}^-_f)\big),
\end{equation*}
\begin{equation*}
\mathcal{N}^-_f=\big(\frac{1}{\rho^+}+\frac{1}{\rho^-})^{-1}(\widetilde{\mathcal{N}_f}-\frac{1}{\rho^+}(\mathcal{N}^+_f-\mathcal{N}^-_f)\big),
\end{equation*}
and
\begin{align*}
\frac{1}{\rho^+}\mathcal{N}^+_f\widetilde{\mathcal{N}_f}^{-1}g^-+\frac{1}{\rho^-}\mathcal{N}^-_f\widetilde{\mathcal{N}_f}^{-1}g^+
=\frac{\rho^+g^+ +\rho^-g^-}{\rho^++\rho^-}-\frac{1}{\rho^++\rho^-}(\mathcal{N}^+_f-\mathcal{N}^-_f)\widetilde{\mathcal{N}_f}^{-1}(g^+-g^-),\\
\mathcal{N}^+_f\widetilde{\mathcal{N}_f}^{-1}\mathcal{N}^-_f=\frac{\rho^+\rho^-}{(\rho^++\rho^-)^2}(\rho^+\mathcal{N}^+_f+\rho^-\mathcal{N}^-_f)-\frac{\rho^+\rho^-}{(\rho^++\rho^-)^2}(\mathcal{N}^+_f-\mathcal{N}^-_f)\widetilde{\mathcal{N}_f}^{-1}(\mathcal{N}^+_f-\mathcal{N}^-_f).
\end{align*}
Accordingly, we obtain that
\begin{equation}\label{system1}
\begin{aligned}
\partial_t \theta=&\frac{1}{\rho^+}\mathcal{N}^+_f p^+-g^+=\frac{1}{\rho^+}\mathcal{N}^+_f\widetilde{\mathcal{N}_f}^{-1}\big(g^+-g^- +\frac{1}{\rho^-}\mathcal{N}^-_f\sigma H(f)\big)-g^+\\
=&-\frac{1}{\rho^+}\mathcal{N}^+_f\widetilde{\mathcal{N}_f}^{-1}g^- - \frac{1}{\rho^-}\mathcal{N}^-_f\widetilde{\mathcal{N}_f}^{-1}g^+
+\frac{\sigma}{\rho^+\rho^-}\mathcal{N}^+_f\widetilde{\mathcal{N}_f}^{-1}\mathcal{N}^-_f H(f)\\
=&-\frac{\rho^+g^+ +\rho^-g^-}{\rho^++\rho^-}+\frac{1}{\rho^++\rho^-}(\mathcal{N}^+-\mathcal{N}^-)\widetilde{\mathcal{N}_f}^{-1}(g^+-g^-)\\
&+\frac{\sigma}{(\rho^++\rho^-)^2}(\rho^+\mathcal{N}^+_f+\rho^-\mathcal{N}^-_f)H(f)-\frac{\sigma}{(\rho^++\rho^-)^2}(\mathcal{N}^+_f-\mathcal{N}^-_f)\widetilde{\mathcal{N}_f}^{-1}(\mathcal{N}^+_f-\mathcal{N}^-_f)H(f)\\
=&\frac{\sigma}{(\rho^++\rho^-)^2}(\rho^+\mathcal{N}^+_f+\rho^-\mathcal{N}^-_f)H(f)\\
&-\frac{2}{\rho^+ +\rho^-}((\rho^+\underline{u}_1^+ +\rho^-\underline{u}_1^-)\partial_1 \theta+(\rho^+\underline{u}_2^+ +\rho^-\underline{u}_2^-)\partial_2 \theta)\\
&-\frac{1}{\rho^+ +\rho^-}\sum_{i,j=1,2}(\rho^+ \underline{u}_i^+\underline{u}_j^+ -\underline{h}_i^+\underline{h}_j^+ +\rho^- \underline{u}_i^-\underline{u}_j^- -\underline{h}_i^-\underline{h}_j^-)\partial_i\partial_j  f\\
&-\frac{\sigma}{(\rho^++\rho^-)^2}(\mathcal{N}^+_f-\mathcal{N}^-_f)\widetilde{\mathcal{N}_f}^{-1}(\mathcal{N}^+_f-\mathcal{N}^-_f)H(f)\\
&+\frac{2}{\rho^+ +\rho^-}(\mathcal{N}_f^+ -\mathcal{N}_f^-)\widetilde{\mathcal{N}_f}^{-1}\mathcal{P}
((\underline{u}_1^+ - \underline{u}_1^-)\partial_1 \theta+(\underline{u}_2^+ - \underline{u}_2^-)\partial_2 \theta)\\
&+\frac{1}{\rho^+ +\rho^-}(\mathcal{N}_f^+ -\mathcal{N}_f^-)\widetilde{\mathcal{N}_f}^{-1}\mathcal{P}(\sum_{i,j=1,2}(\underline{u}_i^+\underline{u}_j^+-\frac{1}{\rho^{+}}\underline{h}_i^+\underline{h}_j^+ -\underline{u}_i^-\underline{u}_j^-+\frac{1}{\rho^{-}}\underline{h}_i^-\underline{h}_j^-)\partial_i\partial_j  f)\\
&-\frac{1}{\rho^+ +\rho^-}\mathbf{N}_f\cdot\big(\underline{\nabla(\rho^+ p_{\textbf{u}^{+},\textbf{u}^{+}}-p_{\textbf{h}^{+},\textbf{h}^{+}})}+\underline{\nabla(\rho^- p_{\textbf{u}^{-},\textbf{u}^{-}}-p_{\textbf{h}^{-},\textbf{h}^{-}})}\big)\\
&+\frac{1}{\rho^+ +\rho^-}(\mathcal{N}_f^+ -\mathcal{N}_f^-)\widetilde{\mathcal{N}_f}^{-1}\mathcal{P}\mathbf{N}_f\cdot\big(\underline{\nabla(\rho^+ p_{\textbf{u}^{+},\textbf{u}^{+}}-p_{\textbf{h}^{+},\textbf{h}^{+}})}-\underline{\nabla(\rho^- p_{\textbf{u}^{-},\textbf{u}^{-}}-p_{\textbf{h}^{-},\textbf{h}^{-}})}\big).
\end{aligned}
\end{equation}
Here $\mathcal{P}:L^2(\mathbb{T}^2)\rightarrow L^2(\mathbb{T}^2)$ denotes the projection operator such that
\begin{equation*}
\mathcal{P}g=g-\langle g\rangle,
\end{equation*}
where $\langle g\rangle:=\int_{\mathbb{T}^2}gdx'$. We can apply the operator $\mathcal{P}$ to some of the terms in \eqref{system1} for the same reason as in \cite{SWZ1}, since it does not change the formulation of this system by the fact that $\mathcal{P}g^\pm=g^\pm$.

From now on until Section 7, we will only discuss the case $\rho^+=\rho^-=1$ for simplicity, and there is no essential difference between this case and the general case.
\subsection{Equations for the Vorticity and Current}
Now we derive the equations for
\begin{equation}
\vom^\pm=\nabla\times\textbf{u}^{\pm},\quad\quad \vxi^\pm=\nabla\times\textbf{h}^{\pm}.
\end{equation}
It follows from \eqref{equa2} by direct calculation that $(\vom^\pm,\vxi^\pm)$ satisfies
\begin{equation}\label{eq:curl}
  \left\{
  	\begin{array}{ll}
  		\partial_t \vom^\pm+\textbf{u}^{\pm}\cdot\nabla\vom^\pm-\textbf{h}^{\pm}\cdot\nabla\vxi^\pm
=\vom^\pm\cdot\nabla\textbf{u}^{\pm}-\vxi^\pm\cdot\nabla\textbf{h}^{\pm}&\text{ in } \Omega_f^\pm,\\
\partial_t \vxi^\pm+\textbf{u}^{\pm}\cdot\nabla\vxi^\pm-\textbf{h}^{\pm}\cdot\nabla\vom^\pm
=\vxi^\pm\cdot\nabla\textbf{u}^{\pm}-\vom^\pm\cdot\nabla\textbf{h}^{\pm}-2\sum_{i=1}^3\nabla u_i^\pm\times\nabla h_i^\pm&\text{ in } \Omega_f^\pm.
  	\end{array}
  \right.
\end{equation}
\subsection{Tangential velocity and magnetic field on $\Gamma^\pm$}
As in \cite{SWZ1}, we need to derive the evolution equations of the following quantities:
\begin{equation}
\mathfrak{a}_i^\pm(t)=\int_{\mathbb{T}^2}u_i^\pm(t,x',\pm 1)dx', \quad \mathfrak{b}_i^\pm(t)=\int_{\mathbb{T}^2}h_i^\pm(t,x',\pm 1)dx'.
\end{equation}
From the fact that $u_3^\pm(t,x',\pm1)\equiv0$, we deduce that for $i=1,2$
\begin{equation*}
\partial_t u_i^\pm+ u_j^\pm\partial_j u_i^\pm- h_j^\pm\partial_j h_i^\pm - \partial_i p^\pm =0 \quad \mathrm{on} \quad \Gamma^\pm .
\end{equation*}
As a result, it holds that
\begin{equation*}
\partial_t \mathfrak{a}_i^\pm+\int_{\Gamma^\pm}(u_j^\pm\partial_j u_i^\pm-h_j^\pm\partial_j h_i^\pm)dx'=0,
\end{equation*}
or equivalently
\begin{equation}\label{eq-beta}
\mathfrak{a}_i^\pm(t)=\mathfrak{a}_i^\pm(0)-\int^t_0\int_{\Gamma^\pm}(u_j^\pm\partial_j u_i^\pm-h_j^\pm\partial_j h_i^\pm)(x',t')dx'dt'.
\end{equation}
Similarly, we have
\begin{equation}\label{eq-gamma}
\mathfrak{b}_i^\pm(t)=\mathfrak{b}_i^\pm(0)-\int^t_0\int_{\Gamma^\pm}(u_j^\pm\partial_j h_i^\pm-h_j^\pm\partial_j u_i^\pm)(x',t')dx'dt'.
\end{equation}
\subsection{Solvability Conditions for the Div-Curl System}
In order to recover the divergence-free velocity field or magnetic field from its curled part, we need to solve the following div-curl system:
\begin{equation}
\left\{\begin{aligned}
&\mathrm{curl}\textbf{u}^{\pm}=\vom^\pm, \quad \mathrm{div}\textbf{u}^{\pm}=g^\pm \quad &\mathrm{in} \quad \Omega_f^{\pm},\\
&\textbf{u}^{\pm}\cdot\textbf{N}_f=\theta \quad &\mathrm{on} \quad \Gamma_f,\\
&\textbf{u}^{\pm}\cdot \mathbf{e_3}=0, \quad \int_{\Gamma^\pm}u_i dx'=\mathfrak{a}_i^\pm(i=1,2) \quad &\mathrm{on} \quad \Gamma^\pm.\\
\end{aligned}
\right.
\end{equation}
The solvability of the above system was obtained by \cite{SWZ1} under the following compatibility conditions:
\begin{itemize}
\item[C1.]$\mathrm{div}\vom^\pm=0\quad \mathrm{in} \quad \Omega_f^{\pm}$,
\item[C2.]$\int_{\Gamma^\pm} \omega_3^\pm dx'=0$,
\item[C3.]$\int_{\mathbb{T}^2}\theta dx'=\mp\int_{\Omega_f^\pm}g^\pm dx$.
\end{itemize}

\section{Energy Estimates for the Linearized System}
In this section, we linearize the equivalent system derived in Section 3 around given functions $(f,\textbf{u}^{\pm},\textbf{h}^{\pm})$ and give the energy estimates for the linearized system. We assume that there exists $T > 0$ such that for any $t\in[0,T]$, there holds
\begin{align*}
	&\|(\underline\vu^{\pm},\underline\vh^{\pm})\|_{W^{1,\infty}}(t)+\|f\|_{W^{2,\infty}}(t)\le L_0,\\
	&\sigma^{\frac{1}{2}}\|f\|_{H^{s+1}}(t)+\|f\|_{H^{s+\frac{1}{2}}}(t)+\|\partial_t f\|_{H^{s-\frac{1}{2}}}(t)+\|\textbf{u}^{\pm}\|_{H^{s}(\Omega_f^{\pm})}(t)+\|\textbf{h}^{\pm}\|_{H^{s}(\Omega_f^{\pm})}(t)
\le L_1,\\
&\|(\underline{\partial_t \vu}^{\pm},\underline{\partial_t \vh}^{\pm})\|_{W^{1,\infty}}(t)\le L_2,\\
&\|f-f_*\|_{H^{s+\frac{1}{2}}}(t)\le \delta_0,\\
&-(1-c_0)\le f(t,x')\le (1-c_0),
\end{align*}
and
\begin{equation*}
\left\{\begin{aligned}
&\mathrm{div}\textbf{u}^{\pm}=\mathrm{div}\textbf{h}^{\pm}=0 \quad &\mathrm{in} \quad \Omega_f^{\pm},\\
&\underline{\textbf{h}}^{\pm}\cdot \textbf{N}_f=0, \quad \underline{\textbf{u}}^{\pm}\cdot \textbf{N}_f=\partial_t f \quad &\mathrm{on} \quad \Gamma_f,\\
&u_3^\pm=h_3^\pm=0 \quad &\mathrm{on} \quad \Gamma^\pm.\\
\end{aligned}
\right.
\end{equation*}
Here $s\ge6$ is a integer and $L_0,L_1,L_2,c_0,\delta_0$ are positive constants.
\subsection{Paralinearization of $\mathcal{N}^\pm_f$ and H}
The third order term $\frac{\sigma}{4}(\mathcal{N}^+_f H(f)+\mathcal{N}^-_f H(f))$ in \eqref{system1} is a fully nonlinear term of $f$, and is difficult to linearize by conventional methods. To overcome this difficulty, we use the paralinearization approach developed in \cite{ABZ1,AlM}. Here we follow the presentation by M\'etivier in \cite{MG}.

\begin{definition}\label{DefSymbol1}
$\forall m\in \mathbb R$, we say that a symbol $a\in\Sigma^m$ if and only if $a$ has the form
\begin{align*}
a=a^{(m)}+a^{(m-1)}
\end{align*}
with
\begin{align*}
a^{(m)}(t,x,\xi)=F(\nabla f(t,x),\xi),
\end{align*}
\begin{align*}
a^{(m-1)}(t,x,\xi)=\sum_{|\alpha|=2}G_\alpha(\nabla f(t,x),\xi)\partial_x^\alpha f(t,x),
\end{align*}
such that:
\begin{itemize}
\item $T_a$ maps real-valued functions to real-valued functions;
\item $F\in C^\infty$ is a real-valued function of $(\zeta,\xi)\in\mathbb R^d\times(\mathbb R^d\backslash0)$, and homogeneous of order m in $\xi$, with a continuous function $C=C(\zeta)>0$ such that
    $F(\zeta,\xi)\ge C(\zeta)|\xi|^m$ for $\forall (\zeta,\xi)\in\mathbb R^d\times(\mathbb R^d\backslash0)$;
\item $G_\alpha$ is a $C^\infty$ complex-valued function of $(\zeta,\xi)\in\mathbb R^d\times(\mathbb R^d\backslash0)$, homogeneous of order $m-1$ in $\xi$.
\end{itemize}
\end{definition}
Let $m\in \mathbb R$ and $A,B$ is two operator of order m, we say $A\sim B$ if $A-B$ is of order $m-2$. We first list some important properties.
\begin{proposition}\label{symbol.pro1}
\cite{ABZ1}Let $m,m'\in \mathbb R$. Then\\
$(1)$ \quad If $a\in \Sigma^m$ and $b\in \Sigma^{m'}$, then $T_a T_b\sim T_{a\sharp b}$ where $a \sharp b\in\Sigma^{m+m'}$ is given by
\begin{align*}
a \sharp b=a^{(m)}b^{(m')}+a^{(m-1)}b^{(m')}+a^{(m)}b^{(m'-1)}+\frac{1}{i}\partial_{\xi}a^{(m)}\cdot\partial_x b^{(m')}.
\end{align*}
$(2)$ \quad If $a\in \Sigma^m$, then $(T_a)^*\sim T_b$ where $b\in\Sigma^m$ is given by
\begin{align*}
b=a^{(m)}+\overline{a^{(m-1)}}+\frac{1}{i}(\partial_x\cdot\partial_\xi)a^{(m)}.
\end{align*}
\end{proposition}
\begin{proof}
From (\ref{composition1}), we can see that for $\rho=2$
\begin{align*}
\|T_{a^{(m)}}T_{b^{(m')}}-T_{a^{(m)}b^{(m;)}+
\frac{1}{i}\partial_{\xi}a^{(m)}\cdot\partial_xb^{m'}}\|_{H^{\mu}\rightarrow H^{\mu-m-m'+2}}\le C(\|\nabla f\|_{W^{2,\infty}}).
\end{align*}
Also, for $\rho=1$, it holds that
\begin{align*}
\|T_{a^{(m)}}T_{b^{(m'-1)}}-T_{a^{(m)}b^{(m'-1)}}\|_{H^{\mu}\rightarrow H^{\mu-m-m'+2}}\le C(\|\nabla f\|_{W^{2,\infty}}),
\end{align*}
\begin{align*}
\|T_{a^{(m-1)}}T_{b^{(m')}}-T_{a^{(m-1)}b^{(m')}}\|_{H^{\mu}\rightarrow H^{\mu-m-m'+2}}\le C(\|\nabla f\|_{W^{2,\infty}}).
\end{align*}
Moreover, (\ref{symbol.est1}) implies that
\begin{align*}
\|T_{a^{(m-1)}}T_{b^{(m'-1)}}\|_{H^{\mu}\rightarrow H^{\mu-m-m'+2}}\le C(\|\nabla f\|_{W^{2,\infty}}).
\end{align*}
The desired conclusion of the first point comes from the Sobolev embeddong $H^{s+1}\subset W^{3,\infty}$. Furthermore, it also shows that $a\sharp b\in\Sigma^{m+m'}$.\par
Similarly, the second point follows from (\ref{adjoint1}).
\end{proof}

Next, we show the paralinearization of the Dirichlet-Neumann operator and the mean curvature operator.
\begin{lemma}\label{decomposition1}\cite{ABZ1}
Assume that $(f,\psi)\in H^{s+1}(\mathbb{T}^2)\times H^{s+\frac{1}{2}}(\mathbb{T}^2)$, then
\begin{align*}
	\mathcal{N}_f^+\psi=T_{\lambda^+}\psi+R^+_1(f,\psi)+r_1^+(f,\psi),\quad\mathcal{N}_f^-\psi=T_{\lambda^-}\psi+R^-_1(f,\psi)+r_1^-(f,\psi).
\end{align*}
Here the symbols $\lambda^\pm=\lambda^{\pm(1)}+\lambda^{\pm(0)}$ are given by
\begin{equation}
\begin{aligned}
&\lambda^{-(1)}=\lambda^{+(1)}=\sqrt{(1+|\nabla f|^2)|\xi|^2-(\nabla f\cdot\xi)^2},\\
&\lambda^{-(0)}=-\overline{\lambda^{+(0)}}=\frac{1+|f|^2}{2\lambda^{-(1)}}\{\mathrm{div}(\alpha^{(1)}\nabla f)+i\partial_\xi\lambda^{-(1)}\cdot\nabla\alpha^{(1)}\}.
\end{aligned}
\end{equation}
with
$$\alpha^{(1)}=\frac{1}{1+|\nabla f|^2}(\lambda^{-(1)}+i\nabla f\cdot\xi).$$
Moreover, we have the estimates
\begin{align*}
	\|R^+_1(f,\psi)\|_{H^{s-\frac{1}{2}}}+\|R^-_1(f,\psi)\|_{H^{s-\frac{1}{2}}}&\le C(\|f\|_{H^{3}},\|\psi\|_{H^{3}})\|f\|_{H^{s+\frac{1}{2}}},\\
	\|r_1^+(f,\psi)\|_{H^{s-\frac{1}{2}}}+\|r_1^-(f,\psi)\|_{H^{s-\frac{1}{2}}}&\le C(\|f\|_{H^{s-\frac{1}{2}}})\|\nabla\psi\|_{H^{s-2}}.
\end{align*}
\end{lemma}
\begin{proof}
	It is well known that the Dirichlet-Neumann operator is an elliptic operator of order 1, and the  expression of its principal symbol $\lambda^{(1)}$ and its subprincipal symbol $\lambda^{(0)}$ is given in \cite{IP}. We claim that the Dirichlet-Neumann operator $\mathcal{N}_f^-$ can be reformulated as
	\begin{align*}
		\mathcal{N}_f^-\psi=T_{\lambda^-}(\psi-T_{\mathcal{B}}f)-T_V\cdot \nabla f+r_1^-(f,\psi),
	\end{align*}
	which satisfies
	\begin{align*}
		\|r_1^-(f,\psi)\|_{H^{s-\frac{1}{2}}}\le C(\|f\|_{H^{s-\frac{1}{2}}})\|\nabla\psi\|_{H^{s-2}}.
	\end{align*}
	Here
	\begin{align*}
		\mathcal{B}:=\frac{\nabla f\cdot\nabla\psi+\mathcal{N}_f^-\psi}{1+|\nabla f|^2}, \quad V:=\nabla\psi-\mathcal{B}\nabla f.
	\end{align*}
	For the proofs of the above claim, we refer the readers to \cite{ABZ1}.

	Then, we let $R^-_1(f,\psi)=-T_{\lambda^-}T_{\mathcal{B}}f-T_V\cdot \nabla f$. By using Proposition \ref{para1} with $m=0,1$, one can see that
	\begin{align*}
\|T_{\lambda}T_{\mathcal{B}}f\|_{H^{s-\frac{1}{2}}}\le C(\|f\|_{H^{3}})\|T_{\mathcal{B}}f\|_{H^{s+\frac{1}{2}}}\le C(\|f\|_{H^{3}},\|\psi\|_{H^{3}})\|f\|_{H^{s+\frac{1}{2}}},\\
\|T_V\cdot \nabla f\|_{H^{s-\frac{1}{2}}}\le C(\|f\|_{H^{3}})\|\nabla f\|_{H^{s-\frac{1}{2}}}\le C(\|f\|_{H^{3}},\|\psi\|_{H^{3}})\|f\|_{H^{s+\frac{1}{2}}},
\end{align*}
which means that $\|R^-_1(f,\psi)\|_{H^{s-\frac{1}{2}}}\le C(\|f\|_{H^{3}},\|\psi\|_{H^{3}})\|f\|_{H^{s+\frac{1}{2}}}$.

The proof for $\mathcal{N}_f^+$ is similar.
\end{proof}

\begin{lemma}\label{lem-p-H}
\cite{ABZ1}Assume that $f\in H^{s+1}(\mathbb{T}^2)$, we shall paralinearize the $H(f)=\mathrm{div}(\frac{\nabla f}{\sqrt{1+|\nabla f|^2}})$ as $H(f)=-T_l f+r_2$, where $l=l^{(2)}+l^{(1)}$ is given by
\begin{equation}
\begin{aligned}
&l^{(2)}=(1+|\nabla f|^2)^{-\frac{1}{2}}(|\xi|^2-\frac{(\nabla f\cdot\xi)^2}{1+|\nabla f|^2}),\\
&l^{(1)}=-\frac{i}{2}(\partial_x\cdot\partial_\xi)l^{(2)},
\end{aligned}
\end{equation}
and $r_2\in L^\infty(0,T;H^{2s-5/2})$ satisfying
\begin{equation}
\|r_2\|_{L^\infty(0,T;H^{2s-5/2})}\le C(\|f\|_{L^\infty(0,T;H^{s+1})}).
\end{equation}
\end{lemma}
\begin{remark}
From the expression of $\lambda^\pm$ and $l$, one can see that $\lambda^\pm\in\Sigma^1$ and $l\in\Sigma^2$, and they are both elliptic symbols.
\end{remark}

Based on the above results, we have
\begin{equation}\label{equation1}
\begin{aligned}
\mathcal{N}_f^+H(f)=-T_{\lambda^+}T_l f+T_{\lambda^+}r_2+R^+_1(f,H(f))+r_1^+(f,H(f)),\\
\mathcal{N}_f^-H(f)=-T_{\lambda^-}T_l f+T_{\lambda^-}r_2+R^-_1(f,H(f))+r_1^-(f,H(f)).
\end{aligned}
\end{equation}
By using Lemma \ref{decomposition1} and Lemma \ref{lem-p-H}, one can see that
\begin{align*}
\|T_{\lambda^+}r_2\|_{H^{s-\frac{1}{2}}}+\|T_{\lambda^-}r_2\|_{H^{s-\frac{1}{2}}}\le C(\|f\|_{H^3})\|r_2\|_{H^{s-\frac{1}{2}}}\le C(\|f\|_{H^{s+\frac{1}{2}}}),\\
\|R^+_1(f,H(f))\|_{H^{s-\frac{1}{2}}}+\|R^-_1(f,H(f))\|_{H^{s-\frac{1}{2}}}\le C (\|f\|_{H^5})\|f\|_{H^{s+\frac{1}{2}}},\\
\|r_1^+(f,H(f))\|_{H^{s-\frac{1}{2}}}+\|r_1^-(f,H(f))\|_{H^{s-\frac{1}{2}}}\le C(\|f\|_{H^{s-\frac{1}{2}}}) \|f\|_{H^{s+1}}.
\end{align*}
Accordingly, we rewrite the three order term as
\begin{align}\label{eq-paraline}
\frac{\sigma}{4}(\mathcal{N}^+_f H(f)+\mathcal{N}^-_f H(f))=-\frac{\sigma}{2}T_{\lambda}T_l f+\frac{\sigma}{4}R,
\end{align}
where
\begin{align*}
	\lambda=\underbrace{\frac{\lambda^{+(1)} +\lambda^{-(1)}}{2}}_{\lambda^{(1)}}+\underbrace{\frac{\lambda^{+(0)} +\lambda^{-(0)}}{2}}_{\lambda^{(0)}},
\end{align*}
and $R=T_{\lambda^+}r_2+R^+_1(f,H(f))+r_1^+(f,H(f))+T_{\lambda^-}r_2+R^-_1(f,H(f))+r_1^-(f,H(f))$ satisfying
\begin{align}\label{eq-est-para-R}
   	\|R\|_{H^{s-\frac{1}{2}}}\le C(\|f\|_{H^{s+\frac{1}{2}}})\|f\|_{H^{s+1}}.
\end{align}
Next, we symmetrize the above paradifferential operator $T_{\lambda}T_l$.
\begin{proposition}\label{symbol.pro4}
\cite{ABZ1}Let $q\in \Sigma^0$ and $\gamma\in \Sigma^{\frac{3}{2}}$ be defined by
\begin{align*}
&q=(1+|\nabla f|^2)^{(-\frac{1}{2})},\\
&\gamma=\underbrace{\sqrt{l^{(2)}\lambda^{(1)}}}_{\gamma^{(\frac{3}{2})}}+\underbrace{
\frac{1}{2i}(\partial_\xi \cdot\partial_x)\sqrt{l^{(2)}\lambda^{(1)}}}_{\gamma^{(\frac{1}{2})}},
\end{align*}
then $T_qT_{\lambda}T_l\sim T_\gamma T_\gamma T_q$ and $T_\gamma\sim (T_\gamma)^*$.
\end{proposition}
\begin{proof}
From Proposition \ref{symbol.pro1}, one can see that proving $T_qT_{\lambda}T_l\sim T_\gamma T_\gamma T_q$ and $T_\gamma\sim (T_\gamma)^*$ is equivalent to showing that
\begin{align*}
q(\lambda\sharp l)+\frac{1}{i}\partial_\xi q\cdot\partial_x(l^{(2)}\lambda^{(1)})
=(\gamma\sharp\gamma)q+\frac{1}{i}\partial_\xi(\gamma^{\frac{3}{2}}\gamma^{\frac{3}{2}})
\cdot\partial_x q,
\end{align*}
and
\begin{align*}
\mathrm{Im} \gamma^{(\frac{1}{2})}=-\frac{1}{2}(\partial_\xi\cdot\partial_x)\gamma^{(\frac{3}{2})},
\end{align*}
where
\begin{align*}
\lambda\sharp l=l^{(2)}\lambda^{(1)}+l^{(1)}\lambda^{(1)}+l^{(2)}\lambda^{(0)}+\frac{1}{i}\partial_\xi
\lambda^{(1)}\cdot\partial_x l^{(2)},
\end{align*}
\begin{align*}
\gamma\sharp\gamma=(\gamma^{(\frac{3}{2})})^2+2\gamma^{(\frac{1}{2})}\gamma^{(\frac{3}{2})}
+\frac{1}{i}\partial_\xi\gamma^{(\frac{3}{2})}\cdot\partial_x\gamma^{(\frac{3}{2})}.
\end{align*}
The above equalities can be easily verified by direct symbolic calculation.
\end{proof}
We introduce the paradifferential operator $T_\beta$ with the symbol
\begin{align*}
\beta:=(\gamma^{(\frac{3}{2})})^{\frac{2s-1}{3}}\in \Sigma^{s-1/2}.
\end{align*}
\begin{lemma}\label{symbol.lem2}
For all $\mu\in\mathbb{R}$, there exists an non-decreasing function C, such that
\begin{align*}
\|[T_\beta,T_\gamma]\|_{H^{\mu+s-1}\rightarrow H^{\mu}}\le C(L_1).
\end{align*}
\end{lemma}
\begin{proof}
From the definition, we have
$$\partial_\xi \beta \cdot \partial_x \gamma^{(\frac{3}{2})}=\partial_\xi\gamma^{(\frac{3}{2})}\cdot \partial_x \beta.$$
Thus, one can arrive at the result of this lemma by using Proposition \ref{symbol.pro1}.
\end{proof}
It is clear that $T_\beta$ is an elliptic operator, whose commutator with $T_\gamma$ is better than $\langle\nabla\rangle^{s-1}$, and we will use it to obtain estimates in Sobolev spaces. This is the reason we introduce such operator.

At the end of this subsection, we present some properties that will be useful in proving energy estimates.
\begin{lemma}\label{symbol.lem1}
For all $\mu\in \mathbb R$, it holds that
\begin{align}
\|T_{\partial_t q}\|_{H^\mu\rightarrow H^\mu}+\|T_{\partial_t \beta}\|_{H^\mu\rightarrow H^{\mu-s+\frac{1}{2}}}+\|T_{\partial_t \gamma}\|_{H^\mu\rightarrow H^{\mu-\frac{3}{2}}}\le C(L_1),
\end{align}
\begin{align}
\|T_{\partial^2_t q}\|_{H^\mu\rightarrow H^\mu}+\|T_{\partial^2_t \beta}\|_{H^\mu\rightarrow H^{\mu-s+\frac{1}{2}}}+\|T_{\partial^2_t \gamma}\|_{H^\mu\rightarrow H^{\mu-\frac{3}{2}}}\le C(L_1,L_2),
\end{align}
\begin{align}
\|T_{\partial_i q}\|_{H^\mu\rightarrow H^\mu}+\|T_{\partial_i \beta}\|_{H^\mu\rightarrow H^{\mu-s+\frac{1}{2}}}+\|T_{\partial_i \gamma}\|_{H^\mu\rightarrow H^{\mu-\frac{3}{2}}}\le C(L_1).
\end{align}
\end{lemma}
\begin{proof}
Recalling the expression of $q, \beta, \gamma$, one can easily verify that $\partial_t q\in \Gamma^0_0$, $\partial_t \beta\in \Gamma^{s-\frac{1}{2}}_0$ and $\partial_t \gamma\in \Gamma^{\frac{3}{2}}_0$. The definition of $\Gamma^m_\rho$ is given in Appendix A. Then, with the help of Proposition \ref{para1} and Sobolev embedding, we deduce that
\begin{align*}
\|T_{\partial_t q}\|_{H^\mu\rightarrow H^\mu}+\|T_{\partial_t \beta}\|_{H^\mu\rightarrow H^{\mu-s+\frac{1}{2}}}+\|T_{\partial_t \gamma}\|_{H^\mu\rightarrow H^{\mu-\frac{3}{2}}}\le C(M_0^0(\partial_t q),M_0^{s-\frac{1}{2}}(\partial_t \beta),M_0^{\frac{3}{2}}(\partial_t \gamma))\\
\le C(\|f\|_{W^{2,\infty}},\|\partial_t f\|_{W^{2,\infty}})\le C(\|f\|_{H^{4}},\|\partial_t f\|_{H^4})\le C(L_1).
\end{align*}
The other estimates can be obtained in the same way.
\end{proof}
\begin{lemma}\label{symbol.lem3}
For all function $a\in H^{s-\frac{1}{2}}$ and $\psi\in H^{s-\frac{3}{2}}$, it holds that
\begin{align*}
\|T_\beta[T_q,a]\psi\|_{L^2}\le C(L_1)\|a\|_{H^{s-\frac{1}{2}}}\|\psi\|_{H^{s-\frac{3}{2}}},\\
\|[T_\beta,a]T_q\psi\|_{L^2}\le C(L_1)\|a\|_{H^{s-\frac{1}{2}}}\|\psi\|_{H^{s-\frac{3}{2}}},\\
\|[T_\gamma,a]\psi\|_{L^2}\le C(L_1)\|a\|_{H^{s-\frac{1}{2}}}\|\psi\|_{H^{\frac{1}{2}}}.
\end{align*}
\begin{proof}
From Lemma \ref{symbol.est2}, we have
$$\|T_\beta[T_q,a]\psi\|_{L^2}\le C(L_1)\|[T_q,a]\psi\|_{H^{s-\frac{1}{2}}}.$$
By using Bony's decomposition, we rewrite $[T_q,a]\psi$ as
\begin{align*}
[T_q,a]\psi=&T_q(a\psi)-aT_q\psi
=T_qT_a\psi+T_qT_\psi a+T_q\mathcal{R}_{\mathcal{B}}(a,\psi)
-T_aT_q\psi-T_{T_q \psi}a -\mathcal{R}_{\mathcal{B}}(a,T_q\psi)\\
=&[T_q,T_a]\psi+T_qT_\psi a-T_{T_q \psi}a+T_q\mathcal{R}_{\mathcal{B}}(a,\psi)-\mathcal{R}_{\mathcal{B}}(a,T_q\psi).
\end{align*}
With the help of Lemma \ref{lem-remainder}, Propositon \ref{adjoint}, Lemma \ref{symbol.est2}, and Proposition \ref{symbol.est3}, one can deduce that
\begin{align*}
&\|[T_q,T_a]\psi\|_{H^{s-\frac{1}{2}}}\le C(\|f\|_{W^{2,\infty}})\|a\|_{W^{1,\infty}}\|\psi\|_{H^{s-\frac{3}{2}}}\le
C(L_1)\|a\|_{H^{s-\frac{1}{2}}}\|\psi\|_{H^{s-\frac{3}{2}}},\\
&\|T_qT_\psi a\|_{H^{s-\frac{1}{2}}}\le C(\|f\|_{W^{1,\infty}})\|\psi\|_{L^\infty}\|a\|_{H^{s-\frac{1}{2}}}\le
C(L_1)\|a\|_{H^{s-\frac{1}{2}}}\|\psi\|_{H^{s-\frac{3}{2}}},\\
&\|T_{T_q \psi}a\|_{H^{s-\frac{1}{2}}}\le C\|T_q \psi\|_{L^\infty}\|a\|_{H^{s-\frac{1}{2}}}\le C\|T_q \psi\|_{H^2}\|a\|_{H^{s-\frac{1}{2}}}\le
C(L_1)\|a\|_{H^{s-\frac{1}{2}}}\|\psi\|_{H^{s-\frac{3}{2}}},\\
&\|T_q\mathcal{R}_{\mathcal{B}}(a,\psi)\|_{H^{s-\frac{1}{2}}}\le C(\|f\|_{W^{1,\infty}})\|\mathcal{R}_{\mathcal{B}}(a,\psi)\|_{H^{s-\frac{1}{2}}}\le C(L_1)\|a\|_{H^{s-\frac{1}{2}}}\|\psi\|_{H^{s-\frac{3}{2}}},\\
&\|\mathcal{R}_{\mathcal{B}}(a,T_q\psi)\|_{H^{s-\frac{1}{2}}}\le C\|a\|_{H^{s-\frac{1}{2}}}\|T_q\psi\|_{H^{s-\frac{3}{2}}}\le C(L_1)\|a\|_{H^{s-\frac{1}{2}}}\|\psi\|_{H^{s-\frac{3}{2}}}.
\end{align*}
Combining the above estimates yields
\begin{align*}
	\|T_{\partial_t q}\|_{H^\mu\rightarrow H^\mu}+\|T_{\partial_t \beta}\|_{H^\mu\rightarrow H^{\mu-s+\frac{1}{2}}}+\|T_{\partial_t \gamma}\|_{H^\mu\rightarrow H^{\mu-\frac{3}{2}}}\le C(L_1).
\end{align*}
The other two inequalities of this lemma can be proved in a similar way.
\end{proof}
\end{lemma}
\subsection{Linearized System of $(f,\theta)$}
In this subsection, we linearize the system of $(f,\theta)$, and give it's energy estimates. From (\ref{system1}) and \eqref{eq-paraline}, we derive the following linearized system
\begin{equation}\label{system2}
\begin{aligned}
\partial_t \bar{f}=&\bar{\theta}\\
\partial_t \bar{\theta}=&-\frac{\sigma}{2}(T_{\lambda }T_{l}\bar{f})
-((\underline{u}_1^+ +\underline{u}_1^-)\partial_1 \bar{\theta}+(\underline{u}_2^+ +\underline{u}_2^-)\partial_2 \bar{\theta})\\
&-\frac{1}{2}\sum_{i,j=1,2}( \underline{u}_i^+\underline{u}_j^+ -\underline{h}_i^+\underline{h}_j^+ + \underline{u}_i^-\underline{u}_j^- -\underline{h}_i^-\underline{h}_j^- )\partial_i\partial_j  \bar{f}+\mathfrak{g}+\frac{\sigma}{4} R,
\end{aligned}
\end{equation}
where
\begin{equation}
\begin{aligned}
\mathfrak{g}=&\frac{1}{2}(\mathcal{N}_f^+ -\mathcal{N}_f^-)\widetilde{\mathcal{N}_f}^{-1}\mathcal{P}(\sum_{i,j=1,2}(\underline{u}_i^+\underline{u}_j^+-\underline{h}_i^+\underline{h}_j^+ -\underline{u}_i^-\underline{u}_j^-+\underline{h}_i^-\underline{h}_j^- )\partial_i\partial_j  f)\\
&+(\mathcal{N}_f^+ -\mathcal{N}_f^-)\widetilde{\mathcal{N}_f}^{-1}\mathcal{P}
((\underline{u}_1^+ - \underline{u}_1^-)\partial_1 \theta+(\underline{u}_2^+ - \underline{u}_2^-)\partial_2 \theta)\\
&-\frac{1}{2}\mathbf{N}_f\cdot\underline{\nabla( p_{\textbf{u}^{+},\textbf{u}^{+}}-p_{\textbf{h}^{+},\textbf{h}^{+}})}
-\frac{1}{2}\mathbf{N}_f\cdot\underline{\nabla( p_{\textbf{u}^{-},\textbf{u}^{-}}-p_{\textbf{h}^{-},\textbf{h}^{-}})}\\
&+\frac{1}{2}(\mathcal{N}_f^+ -\mathcal{N}_f^-)\widetilde{\mathcal{N}_f}^{-1}\mathcal{P}\mathbf{N}_f\cdot
\underline{\nabla(p_{\textbf{u}^{+},\textbf{u}^{+}}-p_{\textbf{h}^{+},\textbf{h}^{+}}-
p_{\textbf{u}^{-},\textbf{u}^{-}}+p_{\textbf{h}^{-},\textbf{h}^{-}})}\\
&-\frac{\sigma}{4}(\mathcal{N}^+_f-\mathcal{N}^-_f)
\widetilde{\mathcal{N}_f}^{-1}(\mathcal{N}^+_f H(f)-\mathcal{N}^-_f H(f))\\
\overset{\text{def}}{=}&\mathfrak{g}_1+\mathfrak{g}_2+\mathfrak{g}_3+\mathfrak{g}_4+\mathfrak{g}_5,
\end{aligned}
\end{equation}
and
\begin{align*}
	R=&T_{\lambda^+}r_2+R^+_1(f,H(f))+r_1^+(f,H(f))+T_{\lambda^-}r_2+R^-_1(f,H(f))+r_1^-(f,H(f)).
\end{align*}
We emphasize that all the paradifferential operators $T_{\lambda^\pm}$, $T_l$, and the remainders $r_1^\pm$, $R_1^\pm$, $r_2$ here are defined by the given function $f$.

Defining
$ w_i=\frac{1}{2}(\underline{u}_i^+ + \underline{u}_i^-)$, and $v_i=\frac{1}{2}(\underline{u}_i^+ -\underline{u}_i^-)$, we rewrite the linearized system as:
\begin{equation}\label{system3}
\begin{aligned}
\partial^2_t \bar{f}=&-\frac{\sigma}{2}(T_{\lambda }T_{l}\bar{f})-2\sum_{i,j=1,2} w_i\partial_i\partial_t\bar{f}\\
&+\frac{1}{2}\sum_{i,j=1,2}( -2 w_i w_j-2v_iv_j +\underline{h}_i^+\underline{h}_j^+ +\underline{h}_i^-\underline{h}_j^- )\partial_i\partial_j  \bar{f}+\mathfrak{g}+\frac{\sigma}{4} R.
\end{aligned}
\end{equation}
We remark that $\int_{\bbT^2}\pa_t \bar f dx'$ may not vanish since we have performed the linearization. For this linear system, we have the following energy estimate.
\begin{proposition}\label{pro-f}
Assume $s\ge 6$, given initial data
$(\bar{\theta}_0,\bar{f}_0)\in H^{s-\frac{1}{2}}\times H^{s+1}(\mathbb{T}^2)$, there exists a unique solution $(\bar{\theta}, \bar{f})\in C\Big([0,T];H^{s-\frac{1}{2}}\times H^{s+1}(\mathbb{T}^2)\Big)$ to the system (\ref{system2}) from $(\bar{\theta}_0,\bar{f}_0)$ so that
\begin{equation*}
\begin{aligned}
&\sup_{t\in [0,T]}(\| \partial_t \bar{f} \|_{H^{s-\frac{1}{2}}}^2+\| \bar{f} \|_{H^{s+\frac{1}{2}}}^2+\sigma\| \bar{f} \|_{H^{s+1}}^2)\\
\le& C(\sigma,L_0)(\| \bar{\theta}_0 \|_{H^{s-\frac{1}{2}}}^2+\| \bar{f}_0 \|_{H^{s+\frac{1}{2}}}^2+\sigma\| \bar{f}_0 \|_{H^{s+1}}^2+\int_0^T \|\mathfrak{g}\|^2_{H^{s-\frac{1}{2}}}+\sigma^2\|R\|^2_{H^{s-\frac{1}{2}}}d\tau)e^{C(c_0,L_1,L_2)T}.
\end{aligned}
\end{equation*}
\end{proposition}
\begin{proof}
We only present the uniform estimates, which ensure the existence and uniqueness of the solution. For convenience, we put all the terms can be bounded by $C(L_0,L_1,L_2)\|\bar{f}\|^2_{H^{s+1}}$ in $R_1$, and terms that can be bounded by $C(L_0,L_1,L_2)(\|\bar{f}\|^2_{H^{s+\frac{1}{2}}}+\|\partial_t\bar{f}\|^2_{H^{s-\frac{1}{2}}})$ in $R_2$. We start the energy estimates from $\frac{1}{2}\partial_t\langle(\partial_t+ w_i\partial_i )T_\beta T_q  \bar{f},(\partial_t+ w_i\partial_i )T_\beta T_q  \bar{f}\rangle$.
\begin{align*}
\frac{1}{2}\partial_t\langle(\partial_t+ w_i\partial_i )T_\beta T_q  \bar{f},(\partial_t+ w_i\partial_i )T_\beta T_q  \bar{f}\rangle
=&\langle\partial^2_tT_\beta T_q  \bar{f},(\partial_t+ w_i\partial_i )T_\beta T_q  \bar{f}\rangle\\
&+\langle w_i\partial_t\partial_iT_\beta T_q  \bar{f},(\partial_t+ w_i\partial_i )T_\beta T_q  \bar{f}\rangle\\
&+\langle(\partial_t w_i)\partial_iT_\beta T_q  \bar{f},(\partial_t+ w_i\partial_i )T_\beta T_q  \bar{f}\rangle\\
\eqdefa&\uppercase\expandafter{\romannumeral1}+\uppercase\expandafter{\romannumeral2}+\uppercase\expandafter{\romannumeral3}.
\end{align*}
It follows from Proposition \ref{para1} and Lemma \ref{symbol.est2} that
\begin{align*}
	\uppercase\expandafter{\romannumeral3}\le C(L_1,L_2)(\|\bar{f}\|^2_{H^{s+\frac{1}{2}}}+\|\partial_t\bar{f}\|^2_{H^{s-\frac{1}{2}}}).
\end{align*}
From (\ref{system3}), we deduce by using Lemma \ref{symbol.lem1} that
\begin{align*}
\uppercase\expandafter{\romannumeral1}=&\langle T_\beta T_q \partial^2_t \bar{f},(\partial_t+ w_i\partial_i )T_\beta T_q  \bar{f}\rangle
+\langle T_{\partial^2_t\beta} T_q  \bar{f}+T_{\beta} T_{\partial^2_t q}  \bar{f},(\partial_t+ w_i\partial_i )T_\beta T_q  \bar{f}\rangle\\
&+2\langle T_{\beta} T_{\partial_t q}\partial_t\bar{f}+T_{\partial_t\beta} T_{ q}\partial_t\bar{f}+T_{\partial_t\beta} T_{\partial_t q}  \bar{f},(\partial_t+ w_i\partial_i )T_\beta T_q  \bar{f}\rangle\\
=&\langle T_\beta T_q \partial^2_t \bar{f},(\partial_t+ w_i\partial_i )T_\beta T_q  \bar{f}\rangle+R_2\\
=&-\frac{\sigma}{2}\langle T_\beta T_q T_\lambda T_l \bar{f},(\partial_t+ w_i\partial_i )T_\beta T_q  \bar{f}\rangle\\
&+\langle T_\beta T_q (-2 w_i\partial_i\partial_t\bar{f}),(\partial_t+ w_i\partial_i )T_\beta T_q  \bar{f}\rangle\\
&+\langle T_\beta T_q [( - w_i\vw_j-v_iv_j +\frac{1}{2}\underline{h}_i^+\underline{h}_j^+ +\frac{1}{2}\underline{h}_i^-\underline{h}_j^- )\partial_i\partial_j  \bar{f}],(\partial_t+ w_i\partial_i )T_\beta T_q  \bar{f}\rangle\\
&+\langle T_\beta T_q (\mathfrak{g}+\frac{\sigma}{4} R),(\partial_t+ w_i\partial_i )T_\beta T_q  \bar{f}\rangle+R_2\\
\eqdefa&I_1+I_2+I_3+I_4+R_2.
\end{align*}
With the help of Proposition \ref{symbol.pro4}, Lemma \ref{symbol.lem2}, and Lemma \ref{symbol.lem3}, we have
\begin{align*}
I_1=&-\frac{\sigma}{2}\langle T_\beta T_q T_\lambda T_l \bar{f},(\partial_t+ w_i\partial_i )T_\beta T_q  \bar{f}\rangle\\
=&-\frac{\sigma}{2}\langle (T_\gamma)^* T_\gamma T_\beta T_q\bar{f},(\partial_t+ w_i\partial_i )T_\beta T_q  \bar{f}\rangle\\
&-\frac{\sigma}{2}\langle \big(((T_\gamma)^*- T_\gamma)T_\gamma T_\beta+T_\gamma[T_\beta,T_\gamma]+[T_\beta,T_\gamma]T_\gamma\big) T_q\bar{f},(\partial_t+ w_i\partial_i )T_\beta T_q  \bar{f}\rangle
\\
&-\frac{\sigma}{2}\langle T_\beta(  T_q T_\lambda T_l-T_\gamma T_\gamma T_q ) \bar{f},(\partial_t+ w_i\partial_i )T_\beta T_q  \bar{f}\rangle\\
=&-\frac{\sigma}{2}\langle  T_\gamma T_\beta  T_q  \bar{f},T_\gamma(\partial_t+ w_i\partial_i )T_\beta T_q  \bar{f}\rangle+R_2\\
=&-\frac{\sigma}{2}\langle  T_\gamma T_\beta  T_q  \bar{f},\partial_tT_\gamma T_\beta T_q  \bar{f}\rangle\\
&+\frac{\sigma}{4}\langle  T_\gamma T_\beta  T_q  \bar{f},(\pa_iw_i)T_\gamma T_\beta T_q  \bar{f}\rangle+\frac{\sigma}{2}\langle  T_\gamma T_\beta  T_q  \bar{f},\big(T_{\pa_t\gamma}+w_iT_{\pa_i\gamma}-[T_\gamma,w_i]\pa_i\big) T_\beta T_q  \bar{f}\rangle+R_2\\
=&-\frac{\sigma}{4}\partial_t\langle  T_\gamma T_\beta  T_q  \bar{f},T_\gamma T_\beta T_q  \bar{f}\rangle+ (\sigma^2+\sigma)R_1+R_2.
\end{align*}
Similarly, it follows from Proposition \ref{symbol.pro4} and Lemma \ref{symbol.lem3} that
\begin{align*}
I_2=-2\langle T_\beta T_q  w_i\partial_i\partial_t\bar{f},(\partial_t+ w_i\partial_i )T_\beta T_q  \bar{f}\rangle
=-2\langle  w_i\partial_i\partial_t T_\beta T_q \bar{f},(\partial_t+ w_i\partial_i )T_\beta T_q  \bar{f}\rangle+R_2.
\end{align*}
Therefore, we have
\begin{align*}
I_2+\uppercase\expandafter{\romannumeral2}=&-\langle  w_i\partial_i\partial_t T_\beta T_q \bar{f},(\partial_t+ w_i\partial_i )T_\beta T_q  \bar{f}\rangle+R_2,\\
=&-\langle  w_i\partial_i\partial_t T_\beta T_q \bar{f},  w_i\partial_i T_\beta T_q  \bar{f}\rangle-\langle  w_i\partial_i\partial_t T_\beta T_q \bar{f},\partial_t T_\beta T_q  \bar{f}\rangle+R_2\\
=&-\langle  \partial_t(w_i\partial_i T_\beta T_q \bar{f}),  w_i\partial_i T_\beta T_q  \bar{f}\rangle+\langle  \frac{1}{2}(\pa_iw_i)\partial_t T_\beta T_q \bar{f}+(\pa_tw_i)\pa_i\partial_t T_\beta T_q \bar{f},\partial_t T_\beta T_q  \bar{f}\rangle+R_2\\
=&-\frac{1}{2}\frac{d}{dt}\| w_i\partial_i T_\beta T_q \bar{f}\|_{L^2}^2+R_2.
\end{align*}
For the same reason, let $a_i= w_i,v_i,\underline{h}_i^+,\underline{h}_i^-$, one can deduce that
\begin{align*}
&\langle T_\beta T_q a_i a_j \partial_i\partial_j  \bar{f},(\partial_t+w_k\partial_k )T_\beta T_q  \bar{f}\rangle\\
=&\langle a_i a_j \partial_i\partial_j T_\beta T_q   \bar{f},(\partial_t+w_k\partial_k )T_\beta T_q  \bar{f}\rangle+R_2\\
=&-\langle  a_j \partial_j T_\beta T_q   \bar{f},a_i \partial_i(\partial_t+w_k\partial_k )T_\beta T_q  \bar{f}\rangle+R_2\\
=&-\langle  a_j \partial_j T_\beta T_q   \bar{f},\partial_t(a_i \partial_iT_\beta T_q  \bar{f})\rangle-\langle  a_j \partial_j T_\beta T_q   \bar{f},w_k\partial_k(a_i \partial_iT_\beta T_q  \bar{f})\rangle+R_2\\
=&-\frac{1}{2}\frac{d}{dt}\| a_i \partial_i T_\beta T_q   \bar{f}\|_{L^2}^2+R_2.
\end{align*}
It follows that
\begin{align*}
I_3=\frac{1}{2}\frac{d}{dt}\|  w_i \partial_i T_\beta T_q   \bar{f}\|_{L^2}^2+\frac{1}{2}\frac{d}{dt}\| v_i \partial_i T_\beta T_q   \bar{f}\|_{L^2}^2-
\frac{1}{4}\frac{d}{dt}\| \underline{h}_i^+ \partial_i T_\beta T_q   \bar{f}\|_{L^2}^2-\frac{1}{4}\frac{d}{dt}\| \underline{h}_i^- \partial_i T_\beta T_q   \bar{f}\|_{L^2}^2
+R_2.
\end{align*}
And obviously, it holds that
\begin{align*}
I_4\le\|\mathfrak{g}\|^2_{H^{s-\frac{1}{2}}}+\sigma^2\| R\|^2_{H^{s-\frac{1}{2}}}+C(L_1,L_2)(\|\bar{f}\|^2_{H^{s+\frac{1}{2}}}+\|\partial_t\bar{f}\|^2_{H^{s-\frac{1}{2}}}).
\end{align*}
Putting the above estimates together, we arrive at
\begin{align}\label{eq-eng-ineq}
\frac{d}{dt}(\sigma E_1(t)+E_2(t))\le C(L_1,L_2)(\sigma\mathcal{E}_1(t)+\mathcal{E}_2(t))+\|\mathfrak{g}\|^2_{H^{s-\frac{1}{2}}}+\sigma^2\| R\|^2_{H^{s-\frac{1}{2}}}.
\end{align}
Here $E_1$, $E_2$, $\mathcal{E}_1$, and $\mathcal{E}_2$ are energy functionals
\begin{align*}
	E_1&=\frac{1}{4}\|T_\gamma T_\beta T_q \bar{f}\|^2_{L^2},\\
	E_2&=\|(\partial_t+ w_i\partial_i )T_\beta T_q  \bar{f}\|^2_{L^2}-\frac{1}{2}\| v_i \partial_i T_\beta T_q   \bar{f}\|_{L^2}^2+\frac{1}{4}\| \underline{h}_i^+ \partial_i T_\beta T_q   \bar{f}\|_{L^2}^2+\frac{1}{4}\| \underline{h}_i^- \partial_i T_\beta T_q   \bar{f}\|_{L^2}^2,
\end{align*}
and
\begin{align*}
	\mathcal{E}_1(t)=\|\bar{f}\|^2_{H^{s+1}},\quad\mathcal{E}_2(t)=\|\bar{f}\|^2_{H^{s+\frac{1}{2}}}+\|\partial_t \bar{f}\|^2_{H^{s-\frac{1}{2}}}.
\end{align*}
By Proposition \ref{para1} and Lemma \ref{symbol.lem3}, one can easily seen that
\begin{align*}
&E_1=\frac{1}{4}\|T_\gamma T_\beta T_q \bar{f}\|^2_{L^2}\le C(L_0)\| T_\beta T_q \bar{f}\|^2_{H^{\frac{3}{2}}}\le C(L_0)\| T_q \bar{f}\|_{H^{s+1}}\le C(L_0)\|  \bar{f}\|^2_{H^{s+1}},\\
&\|(\partial_t+ w_i\partial_i )T_\beta T_q  \bar{f}\|^2_{L^2}\\
\le& C(\|T_{\partial_t\beta} T_q  \bar{f}\|_{L^2}+\|T_{\beta} T_{\partial_tq}  \bar{f}\|_{L^2}+\|T_{\beta} T_{q}  \partial_t\bar{f}\|_{L^2}+\| w_i\|_{L^\infty}\|\partial_iT_\beta T_q  \bar{f}\|_{L^2})\\
\le& C(L_0)(\|\bar{f}\|_{H^{s+\frac{1}{2}}}+\|\partial_t\bar{f}\|_{H^{s-\frac{1}{2}}}),\\
&-\frac{1}{2}\| v_i \partial_i T_\beta T_q   \bar{f}\|_{L^2}+\frac{1}{4}\| \underline{h}_i^+ \partial_i T_\beta T_q   \bar{f}\|_{L^2}+\frac{1}{4}\| \underline{h}_i^- \partial_i T_\beta T_q   \bar{f}\|_{L^2}\le C(L_0)\|\bar{f}\|_{H^{s+\frac{1}{2}}},
\end{align*}
which means that
\begin{align*}
\sigma E_1 +E_2 \le C(L_0)  (\sigma\mathcal{E}_1 +\mathcal{E}_2 ).
\end{align*}
On the other hand, as $T_\gamma$, $T_\beta$, and $T_q$ are all elliptic operators, $\sigma\mathcal{E}_1 +\mathcal{E}_2$ could also be controlled by $\sigma E_1 +E_2$. Indeed, $\gamma\in\Sigma^{\frac{3}{2}}$,
\begin{align}\label{elliptic}
&\gamma^{\frac{3}{2}}(x,\xi)=\sqrt{\lambda^{1}l^{2}}(x,\xi)=(\frac{|\xi|^2+|\nabla f|^2|\xi|^2-(\nabla f\cdot\xi)^2}{1+|\nabla f|^2})^{\frac{3}{4}}
\ge(1+C(L_0))^{-\frac{3}{4}}|\xi|^{\frac{3}{2}},
\end{align}
which have positive lower bound in $x$. We also have $\beta\in\Sigma^{s- \frac{1}{2}}$, and $q\in \Sigma^0$. As a result, it holds that
\begin{align*}
\sigma\|\bar{f}\|^2_{H^{s+1}}\le& C(L_0)\sigma(\|T_q\bar{f}\|^2_{H^{s+1}}+\|\bar{f}\|^2_{L^2})\\
\le& C(L_0)\sigma(\|T_\beta T_q\bar{f}\|^2_{H^{\frac{3}{2}}}+\|\bar{f}\|^2_{L^2})\le
C(L_0)\sigma(\|T_\gamma T_\beta T_q\bar{f}\|^2_{H^{\frac{3}{2}}}+\|\bar{f}\|^2_{L^2}),\\
\|\bar{f}\|^2_{H^{s+\frac{1}{2}}}\le& (\sigma\|\bar{f}\|^2_{H^{s+1}}+C(\sigma)\|\bar{f}\|^2_{L^2})\le
C(L_0)\big(\sigma\|T_\gamma T_\beta T_q\bar{f}\|_{H^{\frac{3}{2}}}+C(\sigma)\|\bar{f}\|_{L^2}\big),\\
\|\partial_t\bar{f}\|^2_{H^{s-\frac{1}{2}}}\le& C(L_0)(\|T_\beta T_q \partial_t \bar{f}\|^2_{L^2}+\|\partial_t \bar{f}\|^2_{L^2})\\
\le&C(L_0)\big(\|(\partial_t+ w_i\partial_i )T_\beta T_q  \bar{f}\|^2_{L^2}+\|\bar{f}\|^2_{H^{s+\frac{1}{2}}}+\|\partial_t \bar{f}\|_{L^2}\big)\\
\le& C(L_0)\big(\|(\partial_t+ w_i\partial_i )T_\beta T_q  \bar{f}\|^2_{L^2}+\sigma\|T_\gamma T_\beta T_q\bar{f}\|_{H^{\frac{3}{2}}}+C(\sigma)\|\bar{f}\|_{L^2}+\|\partial_t\bar{f}\|^2_{L^2}\big).
\end{align*}
Here we use Proposition \ref{symbol.pro2}, Proposition \ref{symbol.pro3}, and the Gagliardo-Nirenberg interpolation inequality, and $C(\sigma)$ is determined by $\sigma$ and $(1+C(L_0))^{-\frac{3}{4}}$. As a conclusion, we have
\begin{align}\label{E2}
\sigma\mathcal{E}_1 +\mathcal{E}_2(t)\le C(L_0)  (\sigma E_1 +E_2 +C(\sigma)\|\bar{f}\|^2_{L^2}+\|\partial_t\bar{f}\|^2_{L^2}).
\end{align}
It is easily seen that
\begin{align*}
\frac{d}{dt}\big(C(\sigma)\|\bar{f}\|^2_{L^2}+\|\partial_t\bar{f}\|^2_{L^2}\big)\le C(\sigma,L_0)(\|\bar{f}\|^2_{H^{s+\frac{1}{2}}}+\|\partial_t \bar{f}\|^2_{H^{s-\frac{1}{2}}})+\|\mathfrak{g}\|^2_{L^2}+\sigma^2\| R\|^2_{L^2}.
\end{align*}
Thus we get by \eqref{eq-eng-ineq} that
\begin{align*}
\sup_{t\in[0,T]}\{\sigma\mathcal{E}_1(t)+\mathcal{E}_2(t)\}\le& C(\sigma,L_0)\Big\{\sigma\|\bar{f}_0\|^2_{H^{s+1}}+\|\bar{f}_0\|^2_{H^{s+\frac{1}{2}}}+\|\bar\theta_0\|^2_{H^{s-\frac{1}{2}}}\\
&+\int_0^T\|\mathfrak{g}\|^2_{H^{s-\frac{1}{2}}}+\sigma^2\| R\|^2_{H^{s-\frac{1}{2}}}d\tau+C(L_1,L_2)\int_0^T\sigma\mathcal{E}_1(\tau)+\mathcal{E}_2(\tau)d\tau\Big\}.
\end{align*}
One can get the desired estimates by Gronwall's inequality.
\end{proof}
\begin{remark}
Notice that $-\frac{1}{2}\| v_i \partial_i T_\beta T_q   \bar{f}\|_{L^2}^2+\frac{1}{4}\| \underline{h}_i^+ \partial_i T_\beta T_q   \bar{f}\|_{L^2}^2+\frac{1}{4}\| \underline{h}_i^- \partial_i T_\beta T_q   \bar{f}\|_{L^2}^2$ may not be positive. We add an extra $C(\sigma)\|\bar{f}\|^2_{L^2}$ to ensure that
\begin{align*}
	&\sigma\|T_\gamma T_\beta T_q \bar{f}\|^2_{L^2}-\frac{1}{2}\| v_i \partial_i T_\beta T_q   \bar{f}\|_{L^2}^2+\frac{1}{4}\| \underline{h}_i^+ \partial_i T_\beta T_q   \bar{f}\|_{L^2}^2+\frac{1}{4}\| \underline{h}_i^- \partial_i T_\beta T_q   \bar{f}\|_{L^2}^2+C(\sigma)\|\bar{f}\|^2_{L^2}\\
	&\ge\|\bar{f}\|^2_{H^{s+\frac{1}{2}}}.
\end{align*}
Here $C(\sigma)$ will get bigger as $\sigma$ gets smaller, and this leads to the above estimates depending on $\sigma$. If the stability condition \eqref{con-sta} holds, we no longer need to introduce $C(\sigma)\|\bar{f}\|^2_{L^2}$, and the energy estimate will not depend on $\sigma$. We will  discuss this kind of problems  in Section 6.
\end{remark}
\begin{lemma}\label{lem-g}
It holds that
\begin{equation*}
\|\mathfrak{g}\|^2_{H^{s-\frac{1}{2}}}+\sigma^2\|R\|^2_{H^{s-\frac{1}{2}}}\le C(L_1).
\end{equation*}
\end{lemma}
\begin{proof}
The estimate of $\|R\|_{H^{s-\frac{1}{2}}}$ is given in \eqref{eq-est-para-R}. By using Proposition \ref{prop:DN-Hs} and Proposition \ref{prop:DN-inverse}, we have
\begin{align*}
	\|\mathfrak{g}_1\|_{H^{s-\frac{1}{2}}}&\le C(L_1)\|(\underline{u}_i^+\underline{u}_j^+  - \underline{u}_i^-\underline{u}_j^- -\underline{h}_i^+\underline{h}_j^+ +\underline{h}_i^-\underline{h}_j^- )\partial_i\partial_j  f\|_{H^{s-\frac{3}{2}}}\\
&\le C(L_1)\|(\underline{\textbf{u}}^{\pm},\underline{\textbf{h}}^{\pm})\|_{H^{s-\frac{3}{2}}}\|f\|_{H^{s+\frac{1}{2}}}\le C(L_1),\\
\|\mathfrak{g}_2\|_{H^{s-\frac{1}{2}}}&\le C(L_1) \|\underline{\textbf{u}}^{\pm}\|_{H^{s-\frac{3}{2}}}\|\theta\|_{{H^{s-\frac{1}{2}}}}\le C(L_1),\\
\|(\mathfrak{g}_3,\mathfrak{g}_4)\|_{H^{s-\frac{1}{2}}}&\le C(L_1)(\|\underline{\nabla( p_{\textbf{u}^{+},\textbf{u}^{+}}-p_{\textbf{h}^{+},\textbf{h}^{+}})}\|_{H^{s-\frac{1}{2}}}
+\|\underline{\nabla( p_{\textbf{u}^{-},\textbf{u}^{-}}-p_{\textbf{h}^{-},\textbf{h}^{-}})}\|_{H^{s-\frac{1}{2}}})\\
&\le C(L_1)(\|{\nabla( p_{\textbf{u}^{+},\textbf{u}^{+}}-p_{\textbf{h}^{+},\textbf{h}^{+}})}\|_{H^{s}(\Omega_f^+)}
+\|{\nabla( p_{\textbf{u}^{-},\textbf{u}^{-}}-p_{\textbf{h}^{-},\textbf{h}^{-}})}\|_{H^{s}(\Omega_f^-)})\\
&\le C(L_1)\|(\textbf{u}^{\pm},\textbf{h}^{\pm})\|_{H^{s}(\Omega_f^\pm)}\le C(L_1),\\
\sigma\|R\|_{H^{s-\frac{1}{2}}}&\le C(L_1)\sigma\|f\|_{H^{s+1}}\le C(L_1).
\end{align*}
This end the proof.
\end{proof}
\subsection{The Linearized System of $(\omega,\xi)$}
From \eqref{eq:curl}, we introduce the following linearized system:
\begin{equation}\label{equa3}
  \left\{
  	\begin{array}{l}
  	\partial_t \bar{\vom}^\pm+\textbf{u}^{\pm}\cdot\nabla\bar{\vom}^\pm-\textbf{h}^{\pm}\cdot\nabla\bar{\vj}^\pm
=\bar{\vom}^\pm\cdot\nabla\textbf{u}^{\pm}-\bar{\vj}^\pm\cdot\nabla\textbf{h}^{\pm},\\
  	\partial_t \bar{\vj}^\pm+\textbf{u}^{\pm}\cdot\nabla\bar{\vj}^\pm-\textbf{h}^{\pm}\cdot\nabla\bar{\vom}^\pm
=\bar{\vj}^\pm\cdot\nabla\textbf{u}^{\pm}-\bar{\vom}^\pm\cdot\nabla\textbf{h}^{\pm}-2\sum^3_{i=1}\nabla u_i^\pm\times\nabla h_i^\pm,	
  	\end{array}
  \right.
\end{equation}
which gives
\begin{equation*}
\begin{aligned}
&\partial_t (\bar{\vom}^\pm+\bar{\vj}^\pm)+\textbf{u}^{\pm}\cdot\nabla(\bar{\vom}^\pm+\bar{\vj}^\pm)
-\textbf{h}^{\pm}\cdot\nabla(\bar{\vom}^\pm+\bar{\vj}^\pm)\\
=&(\bar{\vom}^\pm+\bar{\vj}^\pm)\cdot\nabla\textbf{u}^{\pm}
-(\bar{\vom}^\pm+\bar{\vj}^\pm)\cdot\nabla\textbf{h}^{\pm}-2\sum^3_{i=1}\nabla u_i^\pm\times\nabla h_i^\pm.
\end{aligned}
\end{equation*}
Therefore, we introduce $\varpi^{\pm}=\bar{\vom}^\pm+\bar{\vj}^\pm$ which satisfies
\begin{equation}\label{equa6}
\partial_t \varpi^{\pm}+(\textbf{u}^{\pm}-\textbf{h}^{\pm})\cdot\nabla\varpi^{\pm}=\varpi^{\pm}\cdot\nabla
(\textbf{u}^{\pm}-\textbf{h}^{\pm})-2\nabla u_i^\pm\times\nabla h_i^\pm.
\end{equation}
We define
\begin{align*}
&\frac{dX^\pm(t,x)}{dt}=(\textbf{u}-\textbf{h})^\pm\big(t,X^\pm(t,x)\big), &\quad\quad x\in\Omega_{f_0}^{\pm}.\\
&X^\pm(0,x)=\mathrm{Id}. &\quad\quad x\in\Omega_{f_0}^{\pm},
\end{align*}
where the Id means the identity map. Recalling that $\underline\vh^\pm\cdot\vN_f=0$, one can see that $X^\pm(t,\cdot)$ is a flow map from $\Omega_{f_0}^{\pm}$ to $\Omega_{f(t)}^{\pm}$. Then we have
\begin{equation*}
\frac{d\varpi^{\pm}\big(t,X^\pm(t,x)\big)}{dt}=\big(\varpi^{\pm}\cdot\nabla
(\textbf{u}^{\pm}-\textbf{h}^{\pm})-2\nabla u_i^\pm\times\nabla h_i^\pm\big)\big(t,X^\pm(t,x)\big),\quad\quad x\in\Omega_{f_0}^{\pm}.
\end{equation*}
This is a linear ODE system, and the existence of $\bar{\vom}^\pm+\bar{\vj}^\pm$ follows immediately. So do $\bar{\vom}^\pm-\bar{\vj}^\pm$. Next, we give the energy estimates for $(\bar{\vom}^\pm,\bar{\vj}^\pm)$.
\begin{proposition}\label{pro-curl}
It holds that
\begin{align*}
	\sup_{t\in [0,T]}(\| \bar{\vom}^\pm(t) \|_{H^{s-1}(\Omega_{f}^{\pm})}^2+\| \bar{\vj}^\pm(t) \|_{H^{s-1}(\Omega_{f}^{\pm})}^2)\le (1+\| \bar{\vom}^\pm_0 \|_{H^{s-1}(\Omega_{f_0}^{\pm})}^2
+\| \bar{\vj}^\pm_0|_{H^{s-1}(\Omega_{f_0}^{\pm})}^2)e^{C(L_1)T}.
\end{align*}
\end{proposition}
\begin{proof}
Using the fact that $\underline{\mathbf{u}}^\pm\cdot \mathbf{N}_f=\partial_t f$ and $\underline{\mathbf{h}}^\pm\cdot \mathbf{N}_f=0$, we deduce from (\ref{equa6}) that
\begin{equation*}
\begin{aligned}
\frac{1}{2}\frac{d}{dt}&\int_{\Omega_{f}^{\pm}}|\nabla^{s-1}\varpi^{\pm}(t,x)|^2dx\\
&=\int_{\Omega_{f}^{\pm}}\nabla^{s-1}\varpi^{\pm}\cdot\nabla^{s-1}\partial_t\varpi^{\pm}dx+\frac{1}{2}\int_{\Gamma_f}|\nabla^{s-1}\varpi^{\pm}|^2(\textbf{u}^{\pm}\cdot \textbf{n})d\sigma\\
&\le\int_{\Omega_{f}^{\pm}}\nabla^{s-1}\varpi^{\pm}\cdot\nabla^{s-1}[(\textbf{u}^{\pm}-\textbf{h}^{\pm})\cdot\nabla\varpi^{\pm}]dx
+\frac{1}{2}\int_{\Gamma_f}|\nabla^{s-1}\varpi^{\pm}|^2(\textbf{u}^{\pm}\cdot \textbf{n})d\sigma\\
&\quad+C(L_1)(1+\| \bar{\vom}^\pm(t) \|_{H^{s-1}(\Omega_{f}^{\pm})}^2+\| \bar{\vj}^\pm(t) \|_{H^{s-1}(\Omega_{f}^{\pm})}^2)\\
&\le\frac{1}{2}\int_{\Omega_{f}^{\pm}}(\textbf{u}^{\pm}-\textbf{h}^{\pm})\cdot\nabla
(|\nabla^{s-1}\varpi^{\pm}|^2)dx+\frac{1}{2}\int_{\Gamma_f}|\nabla^{s-1}\varpi^{\pm}|^2(\textbf{u}^{\pm}\cdot \textbf{n})d\sigma\\
&\quad+C(L_1)(1+\| \bar{\vom}^\pm(t) \|_{H^{s-1}(\Omega_{f}^{\pm})}^2+\| \bar{\vj}^\pm(t) \|_{H^{s-1}(\Omega_{f}^{\pm})}^2)\\
&=-\frac{1}{2}\int_{\Omega_{f}^{\pm}}\mathrm{div}(\textbf{u}^{\pm}-\textbf{h}^{\pm})|\nabla^{s-1}\varpi^{\pm}|^2dx
+C(L_1)(1+\| \bar{\vom}^\pm(t) \|_{H^{s-1}(\Omega_{f}^{\pm})}^2+\| \bar{\vj}^\pm(t) \|_{H^{s-1}(\Omega_{f}^{\pm})}^2)\\
&\le C(L_1)(1+\| \bar{\vom}^\pm(t) \|_{H^{s-1}(\Omega_{f}^{\pm})}^2+\| \bar{\vj}^\pm(t) \|_{H^{s-1}(\Omega_{f}^{\pm})}^2).
\end{aligned}
\end{equation*}
Similarly, we have
\begin{equation*}
\begin{aligned}
\frac{1}{2}\frac{d}{dt}&\int_{\Omega_{f}^{\pm}}|\nabla^{s-1}(\bar{\vom}^\pm-\bar{\vj}^\pm)|^2dx
\le C(L_1)(1+\| \bar{\vom}^\pm(t) \|_{H^{s-1}(\Omega_{f}^{\pm})}^2+\| \bar{\vj}^\pm(t) \|_{H^{s-1}(\Omega_{f}^{\pm})}^2).
\end{aligned}
\end{equation*}
The desired estimate follows from the Gronwall's inequality.
\end{proof}
To solve velocity and magnetic field from the vorticity and current $(\bar{\vom}^\pm,\bar{\vj}^\pm)$, we need to verify the following compatibility conditions.
\begin{lemma}
It holds that
\begin{equation*}
\frac{d}{dt}\int_{\Gamma^\pm}\bar{\omega_3}^\pm dx'=0,\quad\quad\frac{d}{dt}\int_{\Gamma^\pm}\bar{\xi_3}^\pm dx'=0.
\end{equation*}
\end{lemma}
\begin{proof}
The proof is straightforward, we refer the readers to \cite{SWZ1}.
\end{proof}
\section{Construction and contraction of the Iteration Map}
Assume that
\begin{equation*}
f_0\in H^{s+1}(\mathbb{T}^2), \quad \textbf{u}_0^{\pm},\textbf{h}_0^{\pm}\in H^{s}(\Omega_{f_0}^{\pm}),
\end{equation*}
which satisfy
\begin{equation*}
\begin{aligned}
& -(1-2c_0)\le f(x')\le (1-2c_0);\\
\end{aligned}
\end{equation*}
for some constant $c_0\in(0,\frac{1}{2})$.

Let $f_*=f_0$, and $\Omega_{*}^{\pm}=\Omega_{f_0}^{\pm}$ be the reference region. The initial data $\big(f_I,(\partial_t f)_I, \vom^\pm_{*I},\vj^\pm_{*I}, \mathfrak{a}^\pm_{iI},\mathfrak{b}^\pm_{iI}\big)$ for the equivalent system is defined as follows:
\begin{equation*}
\begin{aligned}
&f_I=f_0,\quad (\partial_t f)_I=\textbf{u}_0^{\pm}\big(x'.f_0(x')\big)\cdot(-\partial_1 f_0,-\partial_2 f_0,1);\\
&\vom^\pm_{*I}=\mathrm{curl}\textbf{u}_0^{\pm},\quad\vj^\pm_{*I}=\mathrm{curl}\textbf{h}_0^{\pm};\\
&\mathfrak{a}^\pm_{iI}=\int_{\mathbb{T}^2}u_{0i}^\pm (x',\pm 1)dx', \quad\mathfrak{b}^\pm_{iI}=\int_{\mathbb{T}^2}h_{0i}^\pm (x',\pm 1)dx',
\end{aligned}
\end{equation*}
which satisfy
\begin{equation*}
\sigma^{1/2}\|f_I\|_{H^{s+1}}+\|f_I\|_{H^{s+\frac{1}{2}}}+\|(\vom^\pm_{*I},\vj^\pm_{*I})\|_{H^{s-1}(\Omega_{*}^{\pm})}
+\|(\partial_t f)_I\|_{H^{s-\frac{1}{2}}}+|\mathfrak{a}^\pm_{iI}|+|\mathfrak{b}^\pm_{iI}|\le M_0,
\end{equation*}
for some constant $M_0>0$. Then we define the following functional space.
\begin{definition}\label{def1}
Let $s\ge6$ be a integer. Given two constant $M_1,M_2>0$ with $M_1>2M_0$, we define the space $\mathcal{X}_\sigma=\mathcal{X}_\sigma(T,M_1,M_2)$ be the collection of
$(f,\vom^\pm_{*},\vj^\pm_{*},\mathfrak{a}^\pm_{i},\mathfrak{b}^\pm_{i})$ that satisfies
\begin{equation*}
\begin{aligned}
&\big(f(0),\partial_t f(0),\vom^\pm_{*}(0),\vj^\pm_{*}(0),\mathfrak{a}^\pm_{i}(0),\mathfrak{b}^\pm_{i}(0)\big)=\big(f_I,(\partial_t f)_I,\vom^\pm_{*I},\vj^\pm_{*I},\mathfrak{a}^\pm_{iI},\mathfrak{b}^\pm_{iI}\big),\\
&\|f(t,\cdot)-f_*\|_{H^{s-\frac{1}{2}}}\le \delta_0,\\
&\sup_{t\in [0,T]}\big(\sigma^{1/2}\|f\|_{H^{s+1}}(t)+\|f\|_{H^{s+\frac{1}{2}}}(t)\\
&\qquad\qquad\quad+\|(\vom^\pm_{*},\vj^\pm_{*})\|_{H^{s-1}(\Omega_{*}^{\pm})}(t)
+\|\partial_t f\|_{H^{s-\frac{1}{2}}}(t)+|\mathfrak{a}^\pm_{i}|(t)+|\mathfrak{b}^\pm_{i}|(t)\big)\le M_1,\\
&\sup_{t\in [0,T]}\big(\|(\partial_t\vom^\pm_{*},\partial_t\vj^\pm_{*})\|_{H^{s-2}(\Omega_{*}^{\pm})}(t)
+\|\partial_t^2 f\|_{H^{s-2}}(t)+|\partial_t\mathfrak{a}^\pm_{i}|(t)+|\partial_t\mathfrak{b}^\pm_{i}|(t)\big)\le M_2,\\
&\int_{\mathbb{T}^2}\partial_t f(t,x')dx'=0.
\end{aligned}
\end{equation*}
\end{definition}
Next, we will construct an iteration map
\begin{align*}
&\mathcal{F}_\sigma:\mathcal{X}_\sigma(T,M_1,M_2)\to\mathcal{X}_\sigma(T,M_1,M_2),\\
&\mathcal{F}_\sigma(f,\vom^\pm_{*},\vj^\pm_{*},\mathfrak{a}^\pm_{i},\mathfrak{b}^\pm_{i})\eqdefa(\bar{f},\bar{\vom}^\pm_{*},\bar{\vj}^\pm_{*},\bar{\mathfrak{a}}^\pm_{i},\bar{\mathfrak{b}}^\pm_{i}),
\end{align*}
with suitable constants $T,M_1,M_2$.
\subsection{Recover the bulk region, velocity and magnetic field}
We define
\begin{equation*}
\widetilde{\vom}^\pm\eqdefa P_f^{\mathrm{div}}(\vom_{*}^{\pm}\circ \Phi_f^{-1}),\quad
\widetilde{\vj}^\pm\eqdefa P_f^{\mathrm{div}}(\vj_{*}^{\pm}\circ \Phi_f^{-1}),
\end{equation*}
where $\Phi_f^\pm:\Omega_{*}^{\pm}\rightarrow\Omega_{f}^{\pm}$ is the harmonic coordinate map, and
 $P_f^{\mathrm{div}}\vom^{\pm}=\vom^{\pm}-\nabla\phi^{\pm}$ with
 \begin{equation*}
   \left\{
   	\begin{array}{ll}
   		\Delta \phi^{\pm}=\mathrm{div}\vom^{\pm}&\text{ in }\Omega_{f}^{\pm},\\
   		\partial_3 \phi^{\pm}=0&\text{ on }\Gamma^\pm,\\
   		\phi^{\pm}=0&\text{ on }\Gamma_f.
   	\end{array}
   \right.
 \end{equation*}
We introduce the projection operator $P_f^{\mathrm{div}}$ to ensure that $(\widetilde{\vom}^\pm,\widetilde{\vj}^\pm)$ satisfy conditions $(C1)$ and $(C2)$ defined in Section 3.3. It is obvious that
\begin{equation*}
\|(\widetilde{\vom}^\pm,\widetilde{\vj}^\pm)(t)\|_{H^{s-1}(\Omega_{f}^{\pm})}\le C(M_1),
\end{equation*}
\begin{equation*}
\|(\partial_t\widetilde{\vom}^\pm,\partial_t\widetilde{\vj}^\pm)(t)\|_{H^{s-2}(\Omega_{f}^{\pm})}\le C(M_1,M_2).
\end{equation*}
Then, we define $\textbf{u}^{\pm}$ and $\textbf{h}^{\pm}$ as the solution of
\begin{equation*}
\left\{\begin{aligned}
&\mathrm{curl}\textbf{u}^{\pm}=\widetilde{\vom}^\pm, \quad \mathrm{div}\textbf{u}^{\pm}=0 \quad &\mathrm{in} \quad \Omega_f^{\pm},\\
&\textbf{u}^{\pm}\cdot\textbf{N}_f=\partial_t f \quad &\mathrm{on} \quad \Gamma_f,\\
&\textbf{u}^{\pm}\cdot \mathbf{e_3}=0, \quad \int_{\Gamma^\pm}u_i dx'=\mathfrak{a}_i^\pm(i=1,2) \quad &\mathrm{on} \quad \Gamma^\pm,\\
\end{aligned}
\right.
\end{equation*}
\begin{equation*}
\left\{\begin{aligned}
&\mathrm{curl}\textbf{h}^{\pm}=\widetilde{\vj}^\pm, \quad \mathrm{div}\textbf{h}^{\pm}=0 \quad &\mathrm{in} \quad \Omega_f^{\pm},\\
&\textbf{h}^{\pm}\cdot\textbf{N}_f=0 \quad &\mathrm{on} \quad \Gamma_f,\\
&\textbf{h}^{\pm}\cdot \mathbf{e_3}=0, \quad \int_{\Gamma^\pm}h_i dx'=\mathfrak{b}_i^\pm(i=1,2) \quad &\mathrm{on} \quad \Gamma^\pm,\\
\end{aligned}
\right.
\end{equation*}
with initial data
\begin{equation*}
\textbf{u}^{\pm}(0)=\textbf{u}^{\pm}_0,\quad \textbf{h}^{\pm}(0)=\textbf{h}^{\pm}_0.
\end{equation*}
It follows from Proposition \ref{prop:div-curl} that
\begin{equation*}
\|\textbf{u}^{\pm}\|_{H^{s}(\Omega_f^{\pm})}\le C(M_1)(\|\widetilde{\vom}^\pm\|_{H^{s-1}(\Omega_{f}^{\pm})}+\|\partial_t f\|_{H^{s-\frac{1}{2}}}+|\mathfrak{a}^\pm_{1}(t)|+|\mathfrak{a}^\pm_{2}(t)|)\le C(M_1),
\end{equation*}
\begin{equation*}
\| \textbf{h}^{\pm} \|_{H^{s}(\Omega_f^{\pm})}\le C(M_1)(\|\widetilde{\vj}^\pm\|_{H^{s-1}(\Omega_{f}^{\pm})}+|\mathfrak{b}^\pm_{1}(t)|+|\mathfrak{b}^\pm_{2}(t)|)\le C(M_1).
\end{equation*}

Using the same argument to treat $\partial_t\textbf{u}^{\pm}$ and $\partial_t\textbf{h}^{\pm}$, we deduce that
\begin{equation*}
\|\partial_t\textbf{u}^{\pm}\|_{H^{s-1}(\Omega_f^{\pm})}\le C(M_1,M_2),\quad \|\partial_t\textbf{h}^{\pm}\|_{H^{s-1}(\Omega_f^{\pm})}\le C(M_1,M_2),
\end{equation*}
which implies
\begin{align*}
	\|\underline\vu^{\pm}\|_{W^{1,\infty}}(t)&\le \|\underline\vu_0^{\pm}\|_{W^{1,\infty}}+\int_{0}^t\|\underline{\partial_t\vu}^{\pm}\|_{W^{1,\infty}}(t') dt' \le \frac{M_0}{2}+TC(M_1,M_2),\\
	\|\underline\vh^{\pm}\|_{W^{1,\infty}}(t)&\le \frac{M_0}{2}+TC(M_1,M_2).
\end{align*}
Similar argument shows that
\begin{align*}
\|f\|_{W^{2,\infty}}\le \frac{M_0}{2}+TC(M_1).
\end{align*}
Besides, it holds that
\begin{equation*}
\begin{aligned}
&\|f(t)-f_0\|_{L^\infty}\le\|f(t)-f_0\|_{H^{s-\frac{1}{2}}}\le T \|\partial_t f\|_{H^{s-\frac{1}{2}}}\le TM_1.\\
\end{aligned}
\end{equation*}
By choosing $T$ small enough, we have
\begin{equation*}
TM_1\le min\{\delta_0,c_0\},\quad TC(M_1)+TC(M_1,M_2)\le \frac{M_0}{2}, \quad TC(M_1,M_2)\le c_0.
\end{equation*}
Taking $L_0=M_0,L_1=M_1,L_2=C(M_1,M_2)$, we conclude that for $\forall t\in [0,T]$:
\begin{equation*}
\begin{aligned}
&1. -(1-c_0)\le f(t,x')\le (1-c_0),\\
&2. \|(\underline\vu^{\pm},\underline\vh^{\pm})\|_{W^{1,\infty}}(t)+\|f\|_{W^{2,\infty}}(t)\le L_0,\\
&3. \|f(t)-f_*\|_{H^{s-\frac{1}{2}}}\le \delta_0,\\
&4. \sigma^{1/2}\|f\|_{H^{s+1}}(t)+\|f\|_{H^{s+\frac{1}{2}}}(t)+\|\partial_t f\|_{H^{s-\frac{1}{2}}}(t)+\|\textbf{u}^{\pm}\|_{H^{s}(\Omega_f^{\pm})}(t)+\|\textbf{h}^{\pm}\|_{H^{s}(\Omega_f^{\pm})}(t)
\le L_1,\\
&5.\|(\underline{\partial_t \vu}^{\pm},\underline{\partial_t \vh}^{\pm})\|_{W^{1,\infty}}(t)\le L_2.\\
\end{aligned}
\end{equation*}
\subsection{Defining the Iteration Map}
Given $(f,\textbf{u}^{\pm},\textbf{h}^{\pm})$ which is constructed from $(f,\vom^\pm_{*},\vj^\pm_{*},\mathfrak{a}^\pm_{i},\mathfrak{b}^\pm_{i})$. Let $\bar{f_1}$ and $(\bar{\vom}^\pm,\bar{\vj}^\pm)$ be the solutions of the linearized systems \eqref{system2} and \eqref{equa3} with initial data
\begin{equation*}
\big(\bar{f_1}(0),\bar{\theta}(0),\bar{\vom}^\pm(0),\bar{\vj}^\pm(0)\big)=\big(f_0,(\partial_t f)_I,\vom^\pm_{*I},\vj^\pm_{*I}\big).
\end{equation*}
We define
\begin{equation*}
\begin{aligned}
&\bar{\vom}^\pm_*=\bar{\vom}^\pm\circ \Phi_f^\pm, \quad \bar{\vj}^\pm_*=\bar{\vj}^\pm\circ \Phi_f^\pm,\\
&\bar{\mathfrak{a}_i}^\pm(t)={\mathfrak{a}_i}^\pm(0)-\int_{0}^t\int_{\Gamma^\pm}\sum^3_{j=1}(u_j^\pm\partial_j u_i^\pm-h_j^\pm\partial_j h_i^\pm)(x',t')dx'dt',\\
&\bar{\mathfrak{b}_i}^\pm(t)={\mathfrak{b}_i}^\pm(0)-\int_{0}^t\int_{\Gamma^\pm}\sum^3_{j=1}(u_j^\pm\partial_j h_i^\pm-h_j^\pm\partial_j u_i^\pm)(x',t')dx'dt'.
\end{aligned}
\end{equation*}
Then we have the iteration map $\mathcal{F}_\sigma$ as follows
\begin{equation*}
\mathcal{F}_\sigma(f,\vom^\pm_{*},\vj^\pm_{*},\mathfrak{a}^\pm_{i},\mathfrak{b}^\pm_{i})\overset{\text{def}}{=}
(\bar{f},\bar{\vom}^\pm_{*},\bar{\vj}^\pm_{*},\bar{\mathfrak{a}}^\pm_{i},\bar{\mathfrak{b}}^\pm_{i}),
\end{equation*}
where $\bar{f}(t,x')=\mathcal{P}\bar{f_1}(t,x')+\langle f_0\rangle$. Hence, $\langle\bar{f} \rangle=\langle f_0\rangle$ and
$\int_{\mathbb{T}^2}\partial_t\bar{f}(t,x')dx'=0$ for $t\in[0,T]$.
\begin{proposition}
There exists $M_1,M_2,T>0$ depending on $ \delta_0, M_0, \sigma$ so that $\mathcal{F}_\sigma$ is a map from $\mathcal{X}_\sigma(T,M_1,M_2)$ to itself.
\begin{proof}
According to Definition \ref{def1}, the initial conditions are automatically satisfied. It follows from Proposition \ref{pro-f} and Proposition \ref{pro-curl} that
\begin{align*}
\sup_{t\in [0,T]}\big(\sigma^{\frac{1}{2}}\|\bar{f}\|_{H^{s+1}}(t)&+\|\bar{f}\|_{H^{s+\frac{1}{2}}}(t)+\|\bar{\vom}^\pm_{*}\|_{H^{s-1}(\Omega_{*}^{\pm})}(t)\\
&+\|\bar{\vj}^\pm_{*}\|_{H^{s-1}(\Omega_{*}^{\pm})}(t)
+\|\partial_t \bar{f}\|_{H^{s-\frac{1}{2}}}(t)\big)\le C(c_0,\sigma,M_0)e^{C(\sigma,M_1,M_2)T}.	
\end{align*}
We first take $M_1$ large enough such that $ C(c_0,\sigma,M_0)<\frac{M_1}{2}$, then let $T=\frac{1}{10}C(\sigma,M_1,M_2)$ which is till to be determined. Thus, it is straightforward to derive from $(\ref{system2}),(\ref{equa3})$ and the above estimate that
\begin{equation*}
\sup_{t\in [0,T]}\big(\|\partial_t^2 \bar{f}\|_{H^{s-2}}(t)+\|\partial_t\bar{\vom}^\pm_{*}\|_{H^{s-2}(\Omega_{*}^{\pm})}(t)+
\|\partial_t\bar{\vj}^\pm_{*}\|_{H^{s-2}(\Omega_{*}^{\pm})}(t)\big)\le C(M_1).
\end{equation*}
It is clear that
\begin{equation*}
\begin{aligned}
&|\bar{\mathfrak{a}_i}^\pm(t)|+|\bar{\mathfrak{b}_i}^\pm(t)|\le M_0+TC(M_1),\\
&|\partial_t\bar{\mathfrak{a}_i}^\pm(t)|+|\partial_t\bar{\mathfrak{b}_i}^\pm(t)|\le C(M_1),\\
&\|\bar{f}(t)-f_*\|_{H^{s-\frac{1}{2}}}\le \int_0^t\|(\partial_t \bar{f})(t')\|_{H^{s-\frac{1}{2}}}dt'.
\end{aligned}
\end{equation*}
At last we take $M_2=C(M_1)$ and meanwhile $T$ is determined. One can see that all the conditions in Definition \ref{def1} are satisfied. This complete the proof.
\end{proof}
\end{proposition}
\subsection{Contraction of the iteration map} Now, we show that $\mathcal{F}_\sigma$ is contract in $\mathcal{X}_\sigma(T,M_1,M_2)$. Suppose $(f^A,\vom^{\pm A}_{*},\vj^{\pm A}_{*},\mathfrak{a}^{\pm A}_{i},\mathfrak{b}^{\pm A}_{i})$ and
$(f^B,\vom^{\pm B}_{*},\vj^{\pm B}_{*},\mathfrak{a}^{\pm B}_{i},\mathfrak{b}^{\pm B}_{i})$ are two elements
in $\mathcal{X}_\sigma(T,M_1,M_2)$ and
\begin{equation*}
(\bar{f}^{\mathcal C},\bar{\vom}^{\pm\mathcal C}_{*},\bar{\vj}^{\pm\mathcal C}_{*},\bar{\mathfrak{a}}^{\pm\mathcal C}_{i},\bar{\mathfrak{b}}^{\pm\mathcal C}_{i})=\mathcal{F}(f^{\mathcal C},\vom^{\pm\mathcal C}_{*},\vj^{\pm\mathcal C}_{*},\mathfrak{a}^{\pm\mathcal C}_{i},\mathfrak{b}^{\pm\mathcal C}_{i})
\end{equation*}
for $\mathcal C = A,B$. We denote by $g^D$ the difference $g^A-g^B$.
\begin{proposition}\label{prop-constract}
There exists $T=T(c_0,\sigma,\delta_0,M_0)>0$ so that
\begin{equation*}
\begin{aligned}
\bar{E}^D\eqdefa&\sup_{t\in [0,T]}\Big(\sigma^{\frac{1}{2}}\|\bar{f}^D\|_{H^{s-1}}(t)+\|\bar{f}^D\|_{H^{s-\frac{3}{2}}}(t)+\|\bar{\vom}^{\pm D}_{*}\|_{H^{s-3}(\Omega_{*}^{\pm})}(t)+\|\bar{\vj}^{\pm D}_{*}\|_{H^{s-3}(\Omega_{*}^{\pm})}(t)\\
&\quad+\|(\partial_t \bar{f}^D)\|_{H^{s-\frac{5}{2}}}(t)+|\bar{\mathfrak{a}_i}^{\pm D}|(t)+|\bar{\mathfrak{b}_i}^{\pm D}|(t)\Big)\\
\le&\frac{1}{2}\sup_{t\in [0,T]}\Big(\sigma^{\frac{1}{2}}\|f^D\|_{H^{s-1}}(t)+\|f^D\|_{H^{s-\frac{3}{2}}}(t)+\|\vom^{\pm D}_{*}\|_{H^{s-3}(\Omega_{*}^{\pm})}(t)
+\|\vj^{\pm D}_{*}\|_{H^{s-3}(\Omega_{*}^{\pm})}(t)\\
&\quad+\|(\partial_t f^D)\|_{H^{s-\frac{5}{2}}}(t)+|\mathfrak{a}^{\pm D}_{i}|(t)+|\mathfrak{b}^{\pm D}_{i}|(t)\Big)\eqdefa  E^D.
\end{aligned}
\end{equation*}
\end{proposition}
\begin{proof}
By elliptic estimates, we have
\begin{equation*}
\|\Phi^\pm_{f^A}-\Phi^\pm_{f^B}\|_{H^{s-2}(\Omega_{*}^{\pm})}\le C(M_1)\|f^A-f^B\|_{H^{s-\frac{3}{2}}}\le CE^D.
\end{equation*}
For $\mathcal C=A,B$ we define
\begin{equation*}
\textbf{u}^{\pm\mathcal C}_*=\textbf{u}^{\pm}\circ\Phi^\pm_{f^{\mathcal C}},\quad \textbf{h}^{\pm\mathcal C}_*=\textbf{h}^{\pm}\circ\Phi^\pm_{f^{\mathcal C}},
\end{equation*}
and claim that
\begin{equation*}
\|\textbf{u}^{\pm\mathcal C}_*\|_{H^{s-2}(\Omega_{*}^{\pm})}+\|\textbf{h}^{\pm\mathcal C}_*\|_{H^{s-2}(\Omega_{*}^{\pm})}\le CE^D.
\end{equation*}
Indeed, for a vector field $\textbf{v}^{\pm}_*$ defined on $\Omega_{*}^{\pm}$, we define
\begin{equation*}
\begin{aligned}
&\mathrm{curl}_{\mathcal C}\textbf{v}^{\pm}_*=\Big(\mathrm{curl}(\textbf{v}^{\pm}_*\circ(\Phi^\pm_{f^{\mathcal C}})^{-1})\Big)\circ\Phi^\pm_{f^{\mathcal C}},\\
&\mathrm{div}_{\mathcal C}\textbf{v}^{\pm}_*=\Big(\mathrm{div}(\textbf{v}^{\pm}_*\circ(\Phi^\pm_{f^{\mathcal C}})^{-1})\Big)\circ\Phi^\pm_{f^{\mathcal C}}.\\
\end{aligned}
\end{equation*}
Thus, for $\mathcal C=A,B$, it holds that
\begin{equation*}
\left\{\begin{aligned}
&\mathrm{curl}_{\mathcal C}\textbf{u}^{\pm\mathcal C}_*=\widetilde{\vom}_*^{\pm\mathcal C}  \quad &\mathrm{in} \quad \Omega_*^{\pm},\\
&\mathrm{div}_{\mathcal C}\textbf{u}^{\pm\mathcal C}_*=0 \quad &\mathrm{in} \quad \Omega_*^{\pm},\\
&\textbf{u}^{\pm\mathcal C}_*\cdot\textbf{N}_{f^{\mathcal C}}=\partial_t f^{\mathcal C} \quad &\mathrm{on} \quad \Gamma_*,\\
&\textbf{u}^{\pm\mathcal C}\cdot \mathbf{e_3}=0, \quad \int_{\Gamma^\pm}u_i^{\pm\mathcal C} dx'=\mathfrak{a}_i^{\pm\mathcal C}(i=1,2) \quad &\mathrm{on} \quad \Gamma^\pm.
\end{aligned}
\right.
\end{equation*}
Accordingly, we deduce that
\begin{equation*}
\left\{\begin{aligned}
&\mathrm{curl}_A\textbf{u}^{\pm D}_*=\widetilde{\vom}_*^{\pm D}+(\mathrm{curl}_B-\mathrm{curl}_A)\textbf{u}^{\pm B}_*  \quad &\mathrm{in} \quad \Omega_*^{\pm},\\
&\mathrm{div}_A\textbf{u}^{\pm D}_*=(\mathrm{div}_B-\mathrm{div}_A)\textbf{u}^{\pm B}_* \quad &\mathrm{in} \quad \Omega_*^{\pm},\\
&\textbf{u}^{\pm D}_*\cdot\textbf{N}_{f^A}=\partial_t f^D+\textbf{u}^{\pm B}_*\cdot(\textbf{N}_{f^B}-\textbf{N}_{f^A}) \quad &\mathrm{on} \quad \Gamma_*,\\
&\textbf{u}^{\pm D}_*\cdot \mathbf{e_3}=0, \quad \int_{\Gamma^\pm}u_i^{\pm D} dx'=\mathfrak{a}_i^{\pm D}(i=1,2) \quad &\mathrm{on} \quad \Gamma^\pm.\\
\end{aligned}
\right.
\end{equation*}
Direct calculation shows that
\begin{equation*}
\begin{aligned}
&\|(\mathrm{curl}_B-\mathrm{curl}_A)\textbf{u}^{\pm B}_*\|_{H^{s-3}(\Omega_{*}^{\pm})}\le C\|\Phi^\pm_{f^A}-\Phi^\pm_{f^B}\|_{H^{s-2}(\Omega_{*}^{\pm})}\le C \|f^D\|_{H^{s-\frac{3}{2}}}\le CE^D,\\
&\|(\mathrm{curl}_B-\mathrm{curl}_A)\textbf{u}^{\pm B}_*\|_{H^{s-3}(\Omega_{*}^{\pm})}\le CE^D,\\
&\|\textbf{u}^{\pm B}_*\cdot(\textbf{N}_{f^B}-\textbf{N}_{f^A})\|_{H^{s-\frac{5}{2}}}\le CE^D.\\
\end{aligned}
\end{equation*}
Then, we get by Proposition \ref{prop:div-curl} that
\begin{equation*}
\|\textbf{u}^{\pm D}_*\|_{H^{s-2}(\Omega_{*}^{\pm})}\le C(\|\widetilde{\vom}_*^{\pm D}\|_{H^{s-3}(\Omega_{*}^{\pm})}
+\|\partial_t f^D\|_{H^{s-\frac{5}{2}}}+E^D)\le CE^D.
\end{equation*}
Similarly, we have
\begin{equation*}
\|\textbf{h}^{\pm D}_*\|_{H^{s-2}(\Omega_{*}^{\pm})}\le CE^D.
\end{equation*}
Recalling \eqref{system1}, we deduce that
\begin{equation}\label{equa5}
\begin{aligned}
\partial_t \bar{f}_1^D=&\bar{\theta}^D,\\
\partial_t \bar{\theta}^D=&-\frac{\sigma}{2}(T_{\lambda^A }T_{l^A}\bar{f}_1^D)
-\Big((\underline{u}_1^{+A} +\underline{u}_1^{-A})\partial_1 \bar{\theta}^D+(\underline{u}_2^{+A} +\underline{u}_2^{-A})\partial_2 \bar{\theta}^D\Big)\\
&-\frac{1}{2}\sum_{i,j=1,2}\Big( \underline{u}_i^{+A}\underline{u}_j^{+A} -\underline{h}_i^{+A}\underline{h}_j^{+A}+
  \underline{u}_i^{-A}\underline{u}_j^{-A} -\underline{h}_i^{-A}\underline{h}_j^{-A}\Big)\partial_i\partial_j  \bar{f}_1^D+\mathfrak{R},
\end{aligned}
\end{equation}
where
\begin{equation*}
\begin{aligned}
\mathfrak{R}=&-\frac{\sigma}{2}(T_{\lambda^A }T_{l^A}-T_{\lambda^B }T_{l^B})\bar{f}_1^B\\
&-\Big((\underline{u}_1^{+D} +\underline{u}_1^{-D})\partial_1 \bar{\theta}^B+(\underline{u}_2^{+D} +\underline{u}_2^{-D})\partial_2 \bar{\theta}^B\Big)\\
&-\frac{1}{2}\sum_{i,j=1,2}\Big[( \underline{u}_i^{+A}\underline{u}_j^{+A} -\underline{h}_i^{+A}\underline{h}_j^{+A}+
 \underline{u}_i^{-A}\underline{u}_j^{-A} -\underline{h}_i^{-A}\underline{h}_j^{-A})\\
 &-( \underline{u}_i^{+B}\underline{u}_j^{+B} -\underline{h}_i^{+B}\underline{h}_j^{+B}+
  \underline{u}_i^{-B}\underline{u}_j^{-B} -\underline{h}_i^{-B}\underline{h}_j^{-B})\Big]\partial_i\partial_j  \bar{f}_1^B\\
 &+\mathfrak{g}^A-\mathfrak{g}^B,
\end{aligned}
\end{equation*}
and for $\mathcal C=A,B$,
\begin{equation*}
\begin{aligned}
\mathfrak{g}^{\mathcal C}=&+\frac{1}{2}(\mathcal{N}_{f^{\mathcal C}}^+ -\mathcal{N}_{f^{\mathcal C}}^-)\widetilde{\mathcal{N}_{f^{\mathcal C}}}^{-1}
\mathcal{P}\Big(\sum_{i,j=1,2}(\underline{u}_i^{+\mathcal C}\underline{u}_j^{+\mathcal C}
-\underline{h}_i^{+\mathcal C}\underline{h}_j^{+\mathcal C}-
\underline{u}_i^{-\mathcal C}\underline{u}_j^{-\mathcal C}
+\underline{h}_i^{-\mathcal C}\underline{h}_j^{-\mathcal C} )\partial_i\partial_j  f^{\mathcal C}\Big)\\
&+(\mathcal{N}_f^{+\mathcal C} -\mathcal{N}_f^{-\mathcal C})\widetilde{\mathcal{N}_{f^{\mathcal C}}}^{-1}\mathcal{P}
\Big((\underline{u}_1^{+\mathcal C} - \underline{u}_1^{-\mathcal C})\partial_1 \theta^{\mathcal C}+(\underline{u}_2^{+\mathcal C} - \underline{u}_2^{-\mathcal C})\partial_2 \theta^{\mathcal C}\Big)\\
&-\frac{1}{2}\Big(\mathbf{N}_{f^{\mathcal C}}\cdot\underline{\nabla( p_{\textbf{u}^{+\mathcal C},\textbf{u}^{+\mathcal C}}-p_{\textbf{h}^{+\mathcal C},\textbf{h}^{+\mathcal C}})}+\mathbf{N}_{f^{\mathcal C}}\cdot\underline{\nabla( p_{\textbf{u}^{-\mathcal C},\textbf{u}^{-\mathcal C}}-p_{\textbf{h}^{-\mathcal C},\textbf{h}^{-\mathcal C}})}\Big)\\
&+\frac{1}{2}(\mathcal{N}_{f^{\mathcal C}}^+ -\mathcal{N}_{f^{\mathcal C}}^-)\widetilde{\mathcal{N}_{f^{\mathcal C}}}^{-1}\mathcal{P}\mathbf{N}_{f^{\mathcal C}}\cdot
\underline{\nabla(p_{\textbf{u}^{+\mathcal C},\textbf{u}^{+\mathcal C}}-p_{\textbf{h}^{+\mathcal C},\textbf{h}^{+\mathcal C}}-
p_{\textbf{u}^{-\mathcal C},\textbf{u}^{-\mathcal C}}+p_{\textbf{h}^{-\mathcal C},\textbf{h}^{-\mathcal C}})}\\
&+\frac{\sigma}{4}R^{+\mathcal C}.
\end{aligned}
\end{equation*}
It is easy to check that
\begin{equation*}
\|\mathfrak{R}\|_{H^{s-\frac{5}{2}}}\le CE^D.
\end{equation*}
We give the estimate of $\frac{\sigma}{2}(T_{\lambda^A }T_{l^A}-T_{\lambda^B }T_{l^B})\bar{f}^B$ for example, and the other terms can be treated in a similar way.
\begin{align}\label{equaR}
(T_{\lambda^A }T_{l^A}-T_{\lambda^B }T_{l^B})\bar{f}^B=&(T_{\lambda^{(1)A}}-T_{\lambda^{(1)B}})T_{l^{A}}\bar{f}^B+(T_{\lambda^{(0)A}}-T_{\lambda^{(0)B}})T_{l^{A}}\bar{f}^B\\
&+T_{\lambda^B}(T_{l^{(2)A}}-T_{l^{(2)B}})\bar{f}^B+T_{\lambda^B}(T_{l^{(1)A}}-T_{l^{(1)B}})\bar{f}^B.\nonumber
\end{align}
Recalling that $\lambda^{(1)}=\sqrt{(1+|\nabla f|^2)|\xi|^2-(\nabla f\cdot\xi)^2}$, we have
\begin{align*}
	&\lambda^{(1)A}-\lambda^{(1)B}\\
	=&\sqrt{(1+|\nabla f^A|^2)|\xi|^2-(\nabla f^A\cdot\xi)^2}-\sqrt{(1+|\nabla f^B|^2)|\xi|^2-(\nabla f^B\cdot\xi)^2}\\
	=&\frac{(\nabla f^A-\nabla f^B)\cdot(\nabla f^A+\nabla f^B)|\xi|^2-(\nabla f^A-\nabla f^B)\cdot\xi(\nabla f^A+\nabla f^B)\cdot\xi}{\sqrt{(1+|\nabla f^A|^2)|\xi|^2-(\nabla f^A\cdot\xi)^2}+\sqrt{(1+|\nabla f^B|^2)|\xi|^2-(\nabla f^B\cdot\xi)^2}}.
\end{align*}
It follows that
\begin{align*}
\|(T_{\lambda^{(1)A}}-T_{\lambda^{(1)B}})T_{l^{A}}\bar{f}^B\|_{H^{s-\frac{5}{2}}}
\le& C\|\nabla f^A-\nabla f^B\|_{L^\infty}\|T_{l^{A}}\bar{f}^B\|_{H^{s-\frac{3}{2}}}\\
\le& C\|\nabla f^A-\nabla f^B\|_{L^\infty}\|\bar{f}^B\|_{H^{s+\frac{1}{2}}}\le CE^D.
\end{align*}
In this way , one can deduce that
\begin{align*}
	\|(T_{\lambda^A }T_{l^A}-T_{\lambda^B }T_{l^B})\bar{f}^B\|_{H^{s-\frac{5}{2}}}\le CE^D.
\end{align*}

Now, we define
\begin{align*}
\bar{F}_s(\partial_t \bar{f_1}^D,\bar{f_1}^D)\eqdefa&\frac{\sigma}{2}\|T_{\gamma^A} T_{\beta^A} T_{q^A} \bar{f}_1^D\|^2_{L^2}
+\|(\partial_t+ w_i^A\partial_i )T_{\beta^A} T_{q^A}  \bar{f}_1^D\|^2_{L^2}
-\frac{1}{2}\|\underline{u}_i^{+A}\partial_iT_{\beta^A} T_{q^A}  \bar{f}_1\|_{L^2}^2\\
&-\frac{1}{2}\|\underline{u}_i^{-A}\partial_iT_{\beta^A} T_{q^A}  \bar{f}_1\|_{L^2}^2
+\frac{1}{2}\|\underline{h}_i^{+A}\partial_iT_{\beta^A} T_{q^A}  \bar{f}_1\|_{L^2}^2
+\frac{1}{2}\|\underline{h}_i^{-A}\partial_iT_{\beta^A} T_{q^A}  \bar{f}_1\|_{L^2}^2,
\end{align*}
where $\beta^A:=(\gamma^{(\frac{3}{2})A})^{\frac{2s-5}{3}}\in \Sigma^{s-\frac{5}{2}}$, by following the proof of Proposition \ref{pro-f}, one can see that
\begin{equation*}
\frac{d}{dt}\Big(\bar{F}_s(\partial_t \bar{f_1}^D,\bar{f_1}^D)+C(\sigma)\|\bar{f_1}^D\|_{L^2}^2+\|\partial_t\bar{f_1}^D\|_{L^2}^2\Big)\le C(E^D+\bar{E_1}^D),
\end{equation*}
where
\begin{equation*}
\bar{E_1}^D=\sup_{t\in[0,T]}\Big(\sigma\|\bar{f_1}^D(t)\|^2_{H^{s-1}}+\|\bar{f_1}^D(t)\|^2_{H^{s-\frac{3}{2}}}+\|\partial_t\bar{f_1}^D(t)\|^2_{H^{s-\frac{5}{2}}}\Big).
\end{equation*}
As $(f^A,\vom^{\pm A}_{*},\vj^{\pm A}_{*},\mathfrak{a}^{\pm A}_{i},\mathfrak{b}^{\pm A}_{i})\in \mathcal{X}(T,M_1,M_2)$, it holds that
\begin{equation*}
\Big(\sigma\|\bar{f_1}^D(t)\|^2_{H^{s-1}}+\|\bar{f_1}^D(t)\|^2_{H^{s-\frac{3}{2}}}+\|\partial_t\bar{f_1}^D(t)\|^2_{H^{s-\frac{5}{2}}}\Big)\le C\Big(\bar{F}_s(\partial_t \bar{f_1}^D,\bar{f_1}^D)+C(\sigma)\|\bar{f_1}^D\|_{L^2}^2+\|\partial_t\bar{f_1}^D\|_{L^2}^2\Big).
\end{equation*}
Applying Gronwall's inequality, we have
\begin{equation*}
\sup_{t\in[0,T]}\Big(\sigma\|\bar{f_1}^D(t)\|_{H^{s-1}}+\|\bar{f_1}^D(t)\|_{H^{s-\frac{3}{2}}}+\|\partial_t\bar{f_1}^D(t)\|_{H^{s-\frac{5}{2}}}\Big)\le C(e^{CT}-1)E^D,
\end{equation*}
which implies
\begin{equation*}
\sup_{t\in[0,T]}\Big(\sigma\|\bar{f}^D(t)\|_{H^{s-1}}+\|\bar{f}^D(t)\|_{H^{s-\frac{3}{2}}}+\|\partial_t\bar{f}^D(t)\|_{H^{s-\frac{5}{2}}}\Big)\le C(e^{CT}-1)E^D.
\end{equation*}
Similarly,
\begin{equation*}
\sup_{t\in[0,T]}\Big(\|\bar{\vom}_*^D\|_{H^{s-3}(\Omega_{*}^{\pm})}+\|\bar{\vj}_*^D\|_{H^{s-3}(\Omega_{*}^{\pm})}\Big)\le C(e^{CT}-1)E^D.
\end{equation*}
From the equation
\begin{equation*}
\bar{\mathfrak{a}_i}^{\pm\mathcal C}(t)=\bar{\mathfrak{a}_i}^{\pm\mathcal C}(0)-\int_{0}^t\int_{\Gamma^\pm}\sum^3_{j=1}(u_j^{\pm\mathcal C}\partial_j u_i^{\pm\mathcal C}-h_j^{\pm\mathcal C}\partial_j h_i^{\pm\mathcal C})dx'd\tau,
\end{equation*}
we have
\begin{equation*}
|\bar{\mathfrak{a}_i}^{\pm D}(t)|\le |\mathfrak{a}_{iI}^{\pm D}|+TCE^D.
\end{equation*}
Similarly,
\begin{equation*}
|\bar{\mathfrak{b}_i}^{\pm D}(t)|\le |\mathfrak{b}_{iI}^{\pm D}|+TCE^D.
\end{equation*}
As a conclusion, we arrive at
\begin{equation*}
\bar{E}^D\le C(e^{CT}-1+T)E^D.
\end{equation*}
One can achieve the result by taking $T$ small enough.
\end{proof}
\subsection{The limit system}
Proposition \ref{prop-constract} shows that there exists a unique fixed point\\
 $(f,\vom^\pm_*,\vj^\pm_*,\mathfrak{a}_i^\pm,\mathfrak{b}_i^\pm)$ of the map $\mathcal{F}$ in $\mathcal{X}(T,M_1,M_2)$. Now, we will finish the proof of Theorem \ref{thm:1}, and show that one can recover $(\vu^\pm,\vh^\pm,p^\pm)$ from $(f,\vom^\pm_*,\vj^\pm_*,\mathfrak{a}_i^\pm,\mathfrak{b}_i^\pm)$ which is the unique solution to the original system \eqref{equa2}-\eqref{equa:p}. We call $(f,\vom^\pm_*,\vj^\pm_*,\mathfrak{a}_i^\pm,\mathfrak{b}_i^\pm)$ is the solution system of \eqref{equa2}-\eqref{bc-uh}.

From the construction of $\mathcal{F}$, the fixed point $$(f,\vom^\pm,\vj^\pm,\mathfrak{a}_i^\pm,\mathfrak{b}_i^\pm)=(f,\vom^\pm_*\circ\Phi_f^{-1},\vj^\pm_*\circ\Phi_f^{-1},\mathfrak{a}_i^\pm,\mathfrak{b}_i^\pm)$$ satisfies
\begin{align*}
\partial_t f=&\mathcal{P}\theta,	\\
\partial_t \theta=&\frac{\sigma}{2}(\mathcal{N}^+_f H(f)+\mathcal{N}^-_f H(f))
-\frac{1}{2}(\mathcal{N}^+_f-\mathcal{N}^-_f)
\widetilde{\mathcal{N}_f}^{-1}(\mathcal{N}^+_f H(f)-\mathcal{N}^-_f H(f))\\
&-((\underline{u}_1^+ +\underline{u}_1^-)\partial_1 \theta+(\underline{u}_2^+ +\underline{u}_2^-)\partial_2 \theta)-\frac{1}{2}\sum_{i,j=1,2}( \underline{u}_i^+\underline{u}_j^+ -\underline{h}_i^+\underline{h}_j^+ + \underline{u}_i^-\underline{u}_j^- -\underline{h}_i^-\underline{h}_j^- )\partial_i\partial_j  f\\
&+\frac{1}{2}(\mathcal{N}_f^+ -\mathcal{N}_f^-)\widetilde{\mathcal{N}_f}^{-1}\mathcal{P}(\sum_{i,j=1,2}(\underline{u}_i^+\underline{u}_j^+-\underline{h}_i^+\underline{h}_j^+ -\underline{u}_i^-\underline{u}_j^-+\underline{h}_i^-\underline{h}_j^-)\partial_i\partial_j  f)\\
&+(\mathcal{N}_f^+ -\mathcal{N}_f^-)\widetilde{\mathcal{N}_f}^{-1}\mathcal{P}
((\underline{u}_1^+ - \underline{u}_1^-)\partial_1 \theta+(\underline{u}_2^+ - \underline{u}_2^-)\partial_2 \theta)\\
&-\frac{1}{2}\mathbf{N}_f\cdot\underline{\nabla( p_{\textbf{u}^{+},\textbf{u}^{+}}-p_{\textbf{h}^{+},\textbf{h}^{+}})}-\frac{1}{2}\mathbf{N}_f\cdot\underline{\nabla( p_{\textbf{u}^{-},\textbf{u}^{-}}-p_{\textbf{h}^{-},\textbf{h}^{-}})}\\
&+\frac{1}{2}(\mathcal{N}_f^+ -\mathcal{N}_f^-)\widetilde{\mathcal{N}_f}^{-1}\mathcal{P}\mathbf{N}_f\cdot\Big(
\underline{\nabla(p_{\textbf{u}^{+},\textbf{u}^{+}}-p_{\textbf{h}^{+},\textbf{h}^{+}})}-\underline{\nabla(
p_{\textbf{u}^{-},\textbf{u}^{-}}-p_{\textbf{h}^{-},\textbf{h}^{-}})}\Big).
\end{align*}
Here $p_{\textbf{v}^\pm,\textbf{v}^\pm}$ is defined in \eqref{eq-p}, and $(\textbf{u}^{\pm},\textbf{h}^{\pm})$ is the solution to
\begin{equation*}
\left\{\begin{aligned}
&\mathrm{curl}\textbf{u}^{\pm}=P_f^{\mathrm{div}}\vom^\pm, \quad \mathrm{div}\textbf{u}^{\pm}=0 \quad &\mathrm{in} \quad \Omega_f^{\pm};\\
&\textbf{u}^{\pm}\cdot\textbf{N}_f=\partial_t f \quad &\mathrm{on} \quad \Gamma_f;\\
&\textbf{u}^{\pm}\cdot \mathbf{e_3}=0, \quad \int_{\Gamma^\pm}u_i dx'=\mathfrak{a}_i^\pm(i=1,2) \quad &\mathrm{on} \quad \Gamma^\pm;\\
&\partial_t\mathfrak{a}_i^\pm=-\int_{\Gamma^\pm}\sum^3_{j=1}(u_j^\pm\partial_j u_i^\pm-h_j^\pm\partial_j h_i^\pm)dx',
\end{aligned}
\right.
\end{equation*}
and
\begin{equation*}
\left\{\begin{aligned}
&\mathrm{curl}\textbf{h}^{\pm}=P_f^{\mathrm{div}}\vj^\pm, \quad \mathrm{div}\textbf{h}^{\pm}=0 \quad &\mathrm{in} \quad \Omega_f^{\pm};\\
&\textbf{h}^{\pm}\cdot\textbf{N}_f=0 \quad &\mathrm{on} \quad \Gamma_f;\\
&\textbf{h}^{\pm}\cdot \mathbf{e_3}=0, \quad \int_{\Gamma^\pm}h_i dx'=\mathfrak{b}_i^\pm(i=1,2) \quad &\mathrm{on} \quad \Gamma^\pm;\\
&\partial_t\mathfrak{b}_i^{\pm}=-\int_{\Gamma^\pm}\sum^3_{j=1}(u_j^\pm\partial_j h_i^\pm-h_j^\pm\partial_j u_i^\pm)dx',
\end{aligned}
\right.
\end{equation*}
and $(\vom^\pm,\vj^\pm)$ satisfies
\begin{equation*}
\partial_t \vom^\pm+\textbf{u}^{\pm}\cdot\nabla\vom^\pm-\textbf{h}^{\pm}\cdot\nabla\vj^\pm
=\vom^\pm\cdot\nabla\textbf{u}^{\pm}-\vj^\pm\cdot\nabla\textbf{h}^{\pm},
\end{equation*}
\begin{equation*}
\partial_t \vj^\pm+\textbf{u}^{\pm}\cdot\nabla\vj^\pm-\textbf{h}^{\pm}\cdot\nabla\vom^\pm
=\vj^\pm\cdot\nabla\textbf{u}^{\pm}-\vom^\pm\cdot\nabla\textbf{h}^{\pm}-2\sum^3_{i=1}\nabla u_i^\pm\times\nabla h_i^\pm.
\end{equation*}
Next, we will show that the above system is equivalent to the origin system \eqref{equa2}-\eqref{equa:p}. We introduce the pressure $p^\pm$ by
\begin{equation*}
p^{\pm}=\mathcal{H}^\pm\underline{p}^\pm+ p_{\textbf{u}^{\pm},\textbf{u}^{\pm}}-p_{\textbf{h}^{\pm},\textbf{h}^{\pm}},
\end{equation*}
where
\begin{equation*}
\underline{p}^\pm=\widetilde{\mathcal{N}_f}^{-1}(g^+ - g^-\pm\mathcal{N}_f^\mp \sigma H),
\end{equation*}
with
\begin{equation*}
g^{\pm}=2(\underline{u}_1^{\pm}\partial_1 \theta+\underline{u}_2^{\pm}\partial_2 \theta)+
\textbf{N}_f\cdot\underline{\nabla(p_{\textbf{u}^{\pm},\textbf{u}^{\pm}}-p_{\textbf{h}^{\pm},\textbf{h}^{\pm}})}+
\sum_{i,j=1,2}(\underline{u}_i^{\pm}\underline{u}_j^{\pm}-\underline{h}_i^\pm\underline{h}_j^\pm )\partial_i\partial_j  f.
\end{equation*}
Then, for
\begin{equation*}
\textbf{k}^{\pm}\eqdefa\partial_t \textbf{u}^{\pm}+\textbf{u}^{\pm}\cdot\nabla\textbf{u}^{\pm}-\textbf{h}\cdot\nabla\textbf{h}^{\pm}
+\nabla p^{\pm},
\end{equation*}
or
\begin{equation*}
\textbf{k}^{\pm}\eqdefa\partial_t \textbf{h}^{\pm}+\textbf{u}^{\pm}\cdot\nabla\textbf{h}^{\pm}-\textbf{h}^{\pm}\cdot\nabla\textbf{u}^{\pm},
\end{equation*}
one can check following the proof in Section 9 of \cite{SWZ1} that
\begin{equation*}
\left\{\begin{aligned}
&\mathrm{curl}\textbf{k}^{\pm}=0, \quad \mathrm{div}\textbf{k}^{\pm}=0 \quad &\mathrm{in} \quad \Omega_f^{\pm};\\
&\textbf{k}^{\pm}\cdot\textbf{N}_f=0 \quad &\mathrm{on} \quad \Gamma_f;\\
&\textbf{k}^{\pm}\cdot \mathbf{e_3}=0, \quad \int_{\Gamma^\pm}w_i dx'=0(i=1,2) \quad &\mathrm{on} \quad \Gamma^\pm,\\
\end{aligned}
\right.
\end{equation*}
which means that $\textbf{k}^{\pm}\equiv0$, and $(f,\vu^\pm,\vh^\pm,p^\pm)$ is the unique solution to the original system \eqref{equa2}-\eqref{bc-uh}.
\section{Zero surface tension limit}
In the previous section, we have showed that if $(f_0,\vu_0,\vh_0)$ satisfies the assumption of Theorem \ref{thm:1}, there is a unique solution $(f^\sigma,\vu^\sigma,\vh^\sigma)$ of system \eqref{equa2}-\eqref{con-ini} in time $[0.T^\sigma]$. To study the zero surface tension limit, we need to show that if in addition $(f_0,\vu_0,\vh_0)$ satisfies
\begin{align*}
	\Lambda(\vh_0^{\pm},[\vu_0])\ge 2c_0,
\end{align*}
the solution $(f^\sigma,\vu^\sigma,\vh^\sigma)$ can be extended to a $\sigma$ independent time $\bar T$.

Defining energy functionals
\begin{align*}
	G^\sigma_1&=\frac{1}{4}\|T_{\gamma^\sigma} T_{\beta^\sigma} T_{q^\sigma} f^\sigma\|^2_{L^2},\\
	G^\sigma_2&=\|(\partial_t+w^\sigma_i\partial_i )T_{\beta^\sigma} T_{q^\sigma}  f^\sigma\|^2_{L^2}-\frac{1}{2}\| v^\sigma_i \partial_i T_{\beta^\sigma} T_{q^\sigma}   f^\sigma\|_{L^2}^2\\
	&\qquad\qquad\qquad+\frac{1}{4}\| \underline{h}^{\sigma+}_i \partial_i T_{\beta^\sigma} T_{q^\sigma}   f^\sigma\|_{L^2}^2+\frac{1}{4}\| \underline{h}^{\sigma-}_i \partial_i T_{\beta^\sigma} T_{q^\sigma}   f^\sigma\|_{L^2}^2,
\end{align*}
and
\begin{align*}
	\mathcal{G}^\sigma_1(t)=&\frac{1}{2}\|f^\sigma\|^2_{H^{s+1}},\quad\mathcal{G}^\sigma_2(t)=\|f^\sigma\|^2_{H^{s+\frac{1}{2}}}+\|\partial_t f^\sigma\|^2_{H^{s-\frac{1}{2}}},\\
	\mathcal{G}^\sigma(t)=&\| \partial_t f^\sigma \|_{H^{s-\frac{1}{2}}}^2+\| f^\sigma \|_{H^{s+\frac{1}{2}}}^2+\sigma\| f^\sigma \|_{H^{s+1}}^2+\|\textbf{u}^{\sigma\pm}\|_{H^{s}(\Omega_f^{\pm})}+\|\textbf{h}^{\sigma\pm}\|_{H^{s}(\Omega_f^{\pm})},
\end{align*}
we give the following uniform a priori estimate.
\begin{proposition}\label{pro-f1}
Assume $(f^\sigma,\vu^\sigma,\vh^\sigma)$ is the solution of system \eqref{equa2}-\eqref{con-ini} from initial data $(f_0,\vu_0,\vh_0)$ in $[0,T]$ satisfying
\begin{align*}
	\inf_{t\in[0,T]}\Lambda(\vh^{\sigma\pm},[\vu^\sigma])(t)\ge c_0,
\end{align*}
it holds that
\begin{equation}\label{unifromEs}
\begin{aligned}
\sup_{t\in [0,T]}\mathcal{G}^\sigma(t)&\le
C\mathcal{G}^\sigma(0)e^{CT},
\end{aligned}
\end{equation}
where $C$ is a constant independent on $\sigma$.
\end{proposition}
\begin{proof}
Following the procedures of the proof of Proposition \ref{pro-f} and Proposition \ref{pro-curl}, and using Lemma \ref{lem-g}, we can easily deduce that
\begin{align}\label{ineq-energy}
	\frac{d}{dt}\big(\sigma G^\sigma_1+&G^\sigma_2+\|f^\sigma\|^2_{L^2}+\|\partial_tf^\sigma\|^2_{L^2}\\
	&+||\vom^\pm||^2_{H^{s-1}(\Omega_f^\pm)}+||\vxi^\pm||^2_{H^{s-1}(\Omega_f^\pm)}+|\mathfrak{a}_i^\pm|+|\mathfrak{b}_i^\pm|\big)(t)\le P(\mathcal{G}^\sigma)(t).\nonumber
\end{align}
Here $P(\cdot)$ is a polynomial whose coefficients are independent on $\sigma$.

It is clear that
\begin{align*}
	\sigma G^\sigma_1+G^\sigma_2+||\vom^\pm||^2_{H^{s-1}(\Omega_f^\pm)}+||\vxi^\pm||^2_{H^{s-1}(\Omega_f^\pm)}+|\mathfrak{a}_i^\pm|+|\mathfrak{b}_i^\pm|\le C(\mathcal{G}^\sigma)\mathcal{G}^\sigma.
\end{align*}
From the assumption that the Syrovatskij type stability condition holds, $G^\sigma_2(t)$ is positive \cite{SWZ1}, and it holds that
\begin{align}\label{E3}
\mathcal{G}^\sigma_2(t)\le C(c_0,\|(\underline\vu^{\pm},\underline\vh^{\pm})\|_{W^{1,\infty}},\|f\|_{W^{2,\infty}}) (G_2(t)+\|f^\sigma\|^2_{L^2}+\|\partial_tf^\sigma\|^2_{L^2}).
\end{align}
Furthermore, by Proposition \ref{prop:div-curl} we also have
\begin{align*}
	\|\textbf{u}^{\sigma\pm}\|_{H^{s}(\Omega_f^{\pm})}+\|\textbf{h}^{\sigma\pm}\|_{H^{s}(\Omega_f^{\pm})}\le C(\mathcal G^\sigma_2)\big(G^\sigma_2+||\vom^\pm||^2_{H^{s-1}(\Omega_f^\pm)}+||\vxi^\pm||^2_{H^{s-1}(\Omega_f^\pm)}+|\mathfrak{a}_i^\pm|+|\mathfrak{b}_i^\pm|\big).
\end{align*}
Thus $\mathcal G^\sigma$ is equivalent to
\begin{align*}
	\sigma G^\sigma_1+&G^\sigma_2+\|f^\sigma\|^2_{L^2}+\|\partial_tf^\sigma\|^2_{L^2}+||\vom^\pm||^2_{H^{s-1}(\Omega_f^\pm)}+||\vxi^\pm||^2_{H^{s-1}(\Omega_f^\pm)}+|\mathfrak{a}_i^\pm|+|\mathfrak{b}_i^\pm|.
\end{align*}

Combing the above results, one can get the desired estimates by \eqref{ineq-energy} and Gronwall's inequality.
\end{proof}
\begin{remark}\label{rem-ext}
	From Proposition \ref{pro-f1}, we can get uniform estimates for $\|\pa_t\vu^{\sigma\pm}\|_{H^{s-1}(\Omega_f^{\pm})}$, and $\|\pa_t\vh^{\sigma\pm}\|_{H^{s-1}(\Omega_f^{\pm})}$, which determine the value of $\Lambda(\vh^{\sigma\pm},[\vu^\sigma])$. Then, by continuous argument, it is clear that the solution $(f^\sigma,\vu^\sigma,\vh^\sigma)$ gotten in Theorem \ref{thm:1} can be extended to a lifespan $\bar T$ independent of $\sigma$ .
\end{remark}

Thus, Theorem \ref{theo2} can be proofed as follow:
\begin{proof}
The initial data $(f_0,\textbf{u}^{\pm}_0, \textbf{h}^{\pm}_0)$ satisfies all the assumption in Theorem \ref{thm:1}, then there exist a unique solution $(f^\sigma,\textbf{u}^\sigma,\textbf{h}^\sigma)$ of the system \eqref{equa2}-\eqref{bc-uh} in $T^\sigma$. Moreover, from the assumption that
\begin{align*}
	\Lambda(\textbf{h}_0^{\pm},[\textbf{u}_0])\ge 2c_0,
\end{align*}
one can see from Proposition \ref{pro-f1} and Remark \ref{rem-ext} that the solutions can be extended to the one with a lifespan $ T$ independent of $\sigma$. We also denote by $(f^\sigma,\textbf{u}^\sigma,\textbf{h}^\sigma)$ the extended solutions in $[0,T]$ that satisfy
Therefore,
\begin{itemize}
\item[1.]$f^\sigma\in L^\infty([0,T),H^{s+\frac{1}{2}}(\mathbb{T}^2))$,
\item[2.]$\textbf{u}^{\sigma\pm},\textbf{h}^{\sigma\pm}\in L^\infty(0,T;H^{s}(\Omega^\pm_{f^\sigma}))$,
\item[3.]$-(1-c_0)\le f^\sigma\le (1-c_0)$,
\item[4.]$\Lambda(\textbf{h}^{\pm},[\textbf{u}])\ge c_0$.
\end{itemize}
Next, we introduce
\begin{align*}
	\vu^{\sigma\pm}_*=\vu^{\sigma\pm}\circ\Phi^\pm_{f^\sigma}, \quad \vh^{\sigma\pm}_*=\vh^{\sigma\pm}\circ\Phi^\pm_{f^\sigma}.
\end{align*}
It is easily seen that
\begin{align*}
	\sup_{t\in[0,T]}\big(||f^\sigma||^2_{H^{s+\frac{1}{2}}}+||\vu^{\sigma\pm}_*||^2_{H^{s}(\Omega_{f_0})}+||\vh^{\sigma\pm}_*||^2_{H^{s}(\Omega_{f_0})})\big(t)\le C.
\end{align*}
Then there exists a subsequence of $\{(f^\sigma,\vu^{\sigma\pm}_*,\vu^{\sigma\pm}_*)\}$ which converges weakly to some $(f,\vu^{\pm}_*,\vh^{\pm}_*)$ satisfying
\begin{align*}
	\sup_{t\in[0,T]}\big(||f||^2_{H^{s+\frac{1}{2}}}+||\vu^{\pm}_*||^2_{H^{s}(\Omega_{f_0})}+||\vh^{\pm}_*||^2_{H^{s}(\Omega_{f_0})})\big(t)\le C.
\end{align*}
Let $\vu^\pm=\vu^{\pm}_*\circ\Phi^{\pm-1}_{f}$ and $\vh^\pm=\vh^{\pm}_*\circ\Phi^{\pm-1}_{f}$. By a standard compactness argument, we can prove that $(f,\vu,\vh)$ is a solution of the system \eqref{equa2}-\eqref{bc-uh} with $\sigma=0$.
\end{proof}
%\section{Remarks on the two dimensional case}
\section{Further discussion}
In this paper, we study the two phase flow problem with surface tension in the ideal incompressible magnetohydrodynamics. We give a proof of local well-posedness and zero surface tension limit for the case  $\rho^+=\rho^-=1$. The method developed in this paper still works for some general cases.
\subsection{Case $\rho^+\neq\rho^-$}
Recall the evolution equation of $f$ \eqref{system1}. For the general problem that $\rho^+,\rho^->0$, the three order term is
\begin{align*}
	\frac{\sigma}{(\rho^++\rho^-)^2}(\rho^+\mathcal{N}^+_f+\rho^-\mathcal{N}^-_f)H(f)
\end{align*}
instead of  $\frac{\sigma}{4}(\mathcal{N}^+_f H(f)+\mathcal{N}^-_f H(f))$. Define
\begin{align*}
	\lambda^\rho=\rho^+\lambda^++\rho^-\lambda^-=\lambda^{\rho(1)}+\lambda^{\rho(0)},
\end{align*}
where
\begin{align*}
	\lambda^{\rho(1)}=\rho^+\lambda^{+(1)}+\rho^-\lambda^{-(1)},\quad\lambda^{\rho(0)}=\rho^+\lambda^{+(0)}+\rho^-\lambda^{-(0)}.
\end{align*}
Here $\lambda^\pm=\lambda^{\pm(1)}+\lambda^{\pm(0)}$ is defined in Section 4. We have
\begin{align*}
	\frac{\sigma}{(\rho^++\rho^-)^2}(\rho^+\mathcal{N}^+_f+\rho^-\mathcal{N}^-_f)H(f)=-\frac{\sigma}{(\rho^++\rho^-)^2}T_{\lambda^\sigma}T_l f+\frac{\sigma}{(\rho^++\rho^-)^2}R^\rho,
\end{align*}
with remainder term $R^\rho$ satisfying
\begin{align*}
	\|R\|_{H^{s-\frac{1}{2}}}\le C(\|f\|_{H^{s+\frac{1}{2}}})\|f\|_{H^{s+1}}.
\end{align*}
Similar to Proposition \ref{symbol.pro4}, let
\begin{align*}
\gamma^\rho=\underbrace{\sqrt{l^{(2)}\lambda^{\rho(1)}}}_{\gamma^{\rho(\frac{3}{2})}}+\underbrace{\frac{1}{2}\sqrt{\frac{l^{(2)}}{\lambda^{\rho(1)}}}\text{Re}(\lambda^{\rho(0)})+\frac{1}{2i}(\partial_\xi \cdot\partial_x)\sqrt{l^{(2)}\lambda^{\rho(1)}}}_{\gamma^{\rho(\frac{1}{2})}},
\end{align*}
we have $T_qT_{\lambda^\rho}T_l\sim T_{\gamma^\rho} T_{\gamma^\rho} T_q$ and $T_{\gamma^\rho}\sim (T_{\gamma^\rho})^*$.

Then one can use the method introduced in this paper to get similar results in Theorem \ref{thm:1} and Theorem \ref{theo2}. The zero surface tension limit solution is the solution constructed in \cite{LL}.
\subsection{Case $\rho^+=0$} For one fluid problem that there is no fluid and no magnetic in the upper domain, the evolution equation of $f$ is
\begin{align*}
\partial_t^2 f=-2(\underline{u}_1^-\partial_1 \theta+\underline{u}_2^-\partial_2 \theta)+\frac{\sigma}{\rho^-}\mathcal{N}^-_f H(f)-\frac{1}{\rho^-}\textbf{N}_f\cdot
\underline{\nabla(\rho^- p_{\textbf{u}^{-},\textbf{u}^{-}}-p_{\textbf{h}^{-},\textbf{h}^{-}})}\\
-\sum_{i,j=1,2}\underline{u}_i^-\underline{u}_j^-\partial_i\partial_j f+\frac{1}{\rho^{-}}\sum_{i,j=1,2}\underline{h}_i^-\underline{h}_j^-\partial_i\partial_j f.
\end{align*}
For the three order term $\frac{\sigma}{\rho^-}\mathcal{N}^-_f H(f)$, it holds that
\begin{align*}
	\frac{\sigma}{\rho^-}\mathcal{N}^-_f H(f)=-\frac{\sigma}{\rho^-}T_{\lambda^-}T_l f+\frac{\sigma}{\rho^-}R^-,
\end{align*}
and $T_qT_{\lambda^-}T_l\sim T_{\gamma^-} T_{\gamma^-} T_q$ and $T_{\gamma^-}\sim (T_{\gamma^-})^*$,
where
\begin{align*}
	\gamma^-=\underbrace{\sqrt{l^{(2)}\lambda^{-(1)}}}_{\gamma^{-(\frac{3}{2})}}+\underbrace{\frac{1}{2}\sqrt{\frac{l^{(2)}}{\lambda^{-(1)}}}\text{Re}(\lambda^{-(0)})+\frac{1}{2i}(\partial_\xi \cdot\partial_x)\sqrt{l^{(2)}\lambda^{-(1)}}}_{\gamma^{-(\frac{1}{2})}}.
\end{align*}
Then one can get local well-posedness of the one fluid problem by using the method developed in this paper. 
\begin{appendix}
  \section{}
  \subsection{Paradifferential Operator}
In this subsection we will introduce some notations and results about Bony's paradifferential calculus. Here we follow the presentation by M$\acute{e}$tivier in \cite{MG}, for the general theory we refer to \cite{BJ}, \cite{Hormander}, \cite{MG}, \cite{Meyer} and \cite{Tay}.\\
For $\rho\in \mathbb{N}$, we denote $W^{\rho,\infty}(\mathbb{T}^d)$ the Sobolev spaces of $L^\infty$ functions whose derivatives of order $\rho$ are also in $L^\infty$. For
$\rho\in (0,\infty)/\mathbb{N}$, we denote $W^{\rho,\infty}(\mathbb{T}^d)$ the Sobolev spaces of $L^\infty$ functions whose derivatives of order $[\rho]$ are uniformly continuous with exponent $\rho-[\rho]$.
\begin{definition}\label{appDef1}
Given $\rho\ge 0$ and $m\in \mathbb{R}$, denote by $\Gamma^m_\rho(\mathbb{T}^d)$ the space of locally bounded functions $a(x,\xi)$ on $\mathbb{T}^d\times\mathbb{R}^d/\{0\}$, which are $C^\infty$ with respect to $\xi$ for $\xi\neq 0$ and such that, for all $\alpha\in \mathbb{N}^d$ and all $\xi\neq 0$, the function $x\rightarrow\partial^\alpha_\xi a(x,\xi)$ belongs to $W^{\rho,\infty}$ and there exists a constant $C_\alpha$ such that
\begin{align*}
\|\partial^\alpha_\xi a(x,\xi)\|_{W^{\rho,\infty}}:=C_\alpha(1+|\xi|)^{m-|\alpha|}\quad \quad \forall |\xi|\ge\frac{1}{2}
\end{align*}
The seminorm of the symbol is defined by
\begin{align*}
M^m_{\rho}(a):=\sup_{|\alpha|\le\frac{3d}{2}+1+\rho}\sup_{|\xi|\ge\frac{1}{2}}
\|(1+|\xi|)^{|\alpha|-m}\partial^\alpha_\xi a(\cdot,\xi)\|_{W^{\rho,\infty}}
\end{align*}
\end{definition}
Given a symbol a, the paradifferential operator $T_a$ is defined by
\begin{align*}
\widehat{T_au}(\xi):=(2\pi)^{-d}\int\chi(\xi-\eta,\eta)\widehat{a}(\xi-\eta,\eta)\psi(\eta)
\widehat{u}(\eta)d\eta,
\end{align*}
where $\widehat{a}(x,\xi)$ is the Fourier transform of a with respect to the first variable, $\chi(\theta,\xi)\in C^\infty(\mathbb R^d\times \mathbb R^d)$ is an admissible cutoff function: there exists $\varepsilon_1,\varepsilon_2$ such that $0\le\varepsilon_1<\varepsilon_2$ $\psi$ and
\begin{align*}
\chi(\theta,\eta)=1 \quad if \quad |\theta|\le \varepsilon_1|\eta|,\quad \chi(\theta,\eta)=0 \quad if \quad |\theta|\ge \varepsilon_2|\eta|.
\end{align*}
and such that for any $(\theta,\xi)\in\mathbb R^d\times \mathbb R^d$,
\begin{align*}
	|\pa_\theta^\alpha\pa_\eta^\alpha\chi(\theta,\xi)|\le C_{\alpha,\beta}(1+|\eta|)^{-|\alpha|-|\beta|}.
\end{align*}
The cutoff function $\psi(\eta)\in C^\infty(\mathbb R^d)$ satisfies
\begin{align*}
\psi(\eta)=0 \quad for \quad |\eta|\le 1, \quad \psi(\eta)=1 \quad for\quad  |\eta|\ge 2,
\end{align*}
Here we will take the admissible cutoff function $\chi(\theta,\xi)$
\begin{align*}
	\chi(\theta,\xi)=\sum^\infty_{k=0}\zeta_{k-3}(\theta)\phi_k(\eta),
\end{align*}
where $\zeta(\theta)=1$ for $|\theta| \le 1.1, \zeta(\theta)=0$ for $|\theta| \ge 1.9$, and
$$
\left\{\begin{array}{ll}
\xi_{k}(\theta)=\zeta\left(2^{-k} \theta\right) & \text { for } k \in \mathbb{Z} \\
\varphi_{0}=\zeta, \varphi_{k}=\zeta_{k}-\zeta_{k-1} & \text { for } k \ge 1
\end{array}\right.
$$
We also introduce the Littlewood-Paley operators $\Delta_{k}, S_{k}$ defined by
\begin{align*}
	\Delta_{k} u&=\mathcal{F}^{-1}\left(\varphi_{k}(\xi) \hat{u}(\xi)\right) \text { for } k \ge 0, \quad \Delta_{k} u=0 \text { for } k<0, \\
S_{k} u&=\sum_{\ell \le k} \Delta_{\ell} u \text { for } k \in \mathbb{Z}.
\end{align*}
In the case when the function $a$ depends only on the first variable $x$ in $T_{a} u$, we take $\psi=1$. Then $T_{a} u$ is just the usual Bony's paraproduct defined by
\begin{align*}
	T_{a} u=\sum_{k} S_{k-3} a \Delta_{k} u.
\end{align*}
We have the following well-known Bony's decomposition (see \cite{BCD}):
\begin{align*}
	a u=T_{a} u+T_{u} a+\mathcal{R}_{\mathcal{B}}(u, a),
\end{align*}
where the remainder term $\mathcal{R}_{\mathcal{B}}(u, a)$ is defined by
\begin{align*}
\mathcal{R}_{\mathcal{B}}(u, a)=\sum_{|k-\ell| \le 2} \Delta_{k} a \Delta_{\ell} u.
\end{align*}

We list the main features of symbolic calculus for paradifferential operators, the details of proof can be find in \cite{MG}.
\begin{proposition}\label{para1}
Let $m\in \mathbb R$. If $a\in \Gamma^m_0(\mathbb{T}^d)$, then $T_a$ is of order m. Moreover, for all $\mu\in \mathbb{R}$ there exists a constant K such that
\begin{align}\label{symbol.est1}
\|T_a\|_{H^\mu\rightarrow H^{\mu-m}}\le K M^m_0(a).
\end{align}
\end{proposition}
\begin{lemma}\label{lem-remainder}
	If $s>0$ and $s_1,s_2\in \mathbb R$ with $s_1+s_2=s+d/2$, then we have
	\begin{align*}
		||\mathcal{R}_{\mathcal{B}}(u,a)||_{H^s}\le C ||a||_{H^{s_1}}||u||_{H^{s_2}}.
	\end{align*}
\end{lemma}
\begin{proposition}\label{composition}
Let $m\in\mathbb{R}$, and let $\rho>0$. If $a\in\Gamma^m_\rho(\mathbb{T}^d)$, $b\in\Gamma^{m'}_\rho(\mathbb{T}^d)$, then $T_aT_b-T_{a\sharp b}$ is of order $m+m'-\rho$ where
\begin{align*}
a\sharp b=\sum_{|\alpha|<\rho}\frac{1}{i^{|\alpha|}\alpha!}\partial^\alpha_\xi a\partial^\alpha_x b.
\end{align*}
Furthermore, $\forall \mu\in \mathbb{R}$ there exists a constant K such that
\begin{align*}
\|T_aT_b-T_{a\sharp b}\|_{H^{\mu}-H^{\mu-m-m'+\rho}}\le K M^m_\rho(a)M^{m'}_\rho(b).
\end{align*}
\end{proposition}
\begin{proposition}\label{adjoint}
Let $m\in \mathbb{R}$, let$\rho>0$, and let $a\in\Gamma^m_\rho(\mathbb{T}^d)$. Denote by $(T_a)^*$ the adjoint operator of $T_a$ and by $\bar{a}$ the complex conjugate of a. Then $(T_a)^*-T_{a^*}$ is of order $m-\rho$ where
\begin{align*}
a^*=\sum_{|\alpha|<\rho}\frac{1}{i^{|\alpha|}\alpha!}\partial^{\alpha}_\xi\partial^{\alpha}_x \bar{a}.
\end{align*}
Furthermore, $\forall \mu$ there exists a constant K such that
\begin{align*}
\|(T_a)^*-T_a^*\|_{H^\mu-H^{\mu-m+\rho}}\le K M^m_\rho(a).
\end{align*}
\end{proposition}
If $a=a(x)$ is independent of $\xi$, then $T_a$ is called a paraproduct. From Propositon \ref{adjoint} and Proposition \ref{composition}, we can get:
\begin{itemize}
\item If $a\in H^\alpha(\mathbb{T}^d)$ and $b\in H^\beta(\mathbb{T}^d)$ with $\alpha>\frac{d}{2}$, $\beta>\frac{d}{2}$, then
\begin{align}\label{composition1}
T_aT_b-T_{ab} \quad is\ of\  order \quad -(\min\{\alpha,\beta\}-\frac{d}{2}).
\end{align}
\item If $a\in H^\alpha(\mathbb{T}^d)$ with $\alpha>\frac{d}{2}$, then
\begin{align}\label{adjoint1}
(T_a)^*-T_{\bar{a}} \quad is\ of\ order\quad -(\alpha-\frac{d}{2}).
\end{align}
\end{itemize}
\begin{lemma}\label{symbol.est2}
Let $m>0$. If $a\in H^{\frac{d}{2}-m}(\mathbb{T}^d)$ and $u\in H^\mu(\mathbb{T}^d)$, then $T_a u\in H^{\mu-m}(\mathbb{T}^d)$. Moreover,
\begin{align*}
\|T_a u\|_{H^{\mu-m}}\le K\|a\|_{H^{\frac{d}{2}-m}}\|u\|_{H^\mu},
\end{align*}
where the constant K is independent of a and u.
\end{lemma}
\begin{proposition}\label{symbol.est3}
Let $\alpha,\beta\in \mathbb{R}$ such that $\alpha>\frac{d}{2}$, $\beta>\frac{d}{2}$, then
\begin{itemize}
\item $\forall F\in C^\infty$, if $a\in H^\alpha(\mathbb{T}^d)$, then
\begin{align*}
F(a)-F(0)-T_{F'(a)}a\in H^{2\alpha-\frac{d}{2}}(\mathbb{T}^d).
\end{align*}
\item If $a\in H^\alpha(\mathbb{T}^d)$ and $b\in H^\beta(\mathbb{T}^d)$, then $ab-T_ab-T_ba\in H^{\alpha+\beta-\frac{d}{2}}(\mathbb{T}^d)$. Moreover,
\begin{align*}
\|ab-T_ab-T_ba\|_{H^{\alpha+\beta-\frac{d}{2}}}\le K\|a\|_{H^\alpha(\mathbb{T}^d)}\|b\|_{H^\beta(\mathbb{T}^d)}
\end{align*}
where the constant K is independent of a, b.
\end{itemize}
\end{proposition}
We also represent here some nonlinear estimates in Sobolev spaces
\begin{itemize}
\item If $u_j\in H^{s_j}(\mathbb{T}^d),\ j=1,2$, and $s_1+s_2>0$, then $u_1u_2\in H^{s_0}(\mathbb{T}^d)$; if
\begin{align*}
s_0\le s_j, \quad j=1,2, \quad and \quad s_0\le s_1+s_2-\frac{d}{2}, \quad then
\end{align*}
\begin{align*}
\|u_1u_2\|_{H^{s_0}}\le K \|u_1\|_{H^{s_1}}\|u_2\|_{H^{s_2}},
\end{align*}
where the last inequality is strict if $s_1$ or $s_2$ or $-s_0$ is equal to 0.
\item If $u\in H^s(\mathbb{T}^d)$ with $s>\frac{d}{2}$, then $\forall F\in C^\infty$ vanishing at the origin,
\begin{align*}
\|F(u)\|_{H^s}\le C(\|u\|_{H^s}),
\end{align*}
where the constant C is non-decreasing and depending only on F.
\end{itemize}
Recall the definition \ref{DefSymbol1}, here we list some properties of the symbol
$a\in \sum^m$(see \cite{ABZ1}):
\begin{proposition}\label{symbol.pro2}
Let $m\in \mathbb R$ and $\mu\in \mathbb R$. Then there exists a function C such that for all symbols $a\in\Sigma^m$ and all $t\in [0,T]$,
\begin{align*}
\|T_{a(t)}u\|_{H^{\mu-m}}\le C(\|f\|_{H^{s-1}})\|u\|_{H^\mu}.
\end{align*}
\end{proposition}
\begin{proposition}\label{symbol.pro3}
Let $m\in \mathbb R$ and $\mu\in \mathbb R$. Then there exists a function C such that for all symbols $a\in\Sigma^m$ and all $t\in [0,T]$,
\begin{align*}
\|u\|_{H^{\mu+m}}\le C(\|f\|_{H^{3}})(\|T_{a(t)}u\|_{H^\mu}+\|u\|_{L^2}).
\end{align*}
\end{proposition}
\subsection{Div-Curl system}
From Section 5 of \cite{SWZ1}, we know that for each div-curl system
\begin{equation}\label{eq:div-curl}
  \left\{
  	\begin{array}{ll}
  		\curl \vu =\vom,\quad\div \vu=g& \text{ in }\quad\Om_f^+,\\
\vu\cdot\vN_f =\vartheta& \text{ on}\quad \Gamma_f,   \\
\vu\cdot\ve_3 = 0,\quad \int_{\bbT^2} u_i dx'=\alpha_i (i=1,2)& \text{ on}\quad\Gamma^{+}.
  	\end{array}
  \right.
\end{equation}
with $f\in H^{s+\frac{1}{2}}(\bbT^2)$ for $s\ge 2$ and satisfying
\begin{align*}
	-(1-c_0)\le f\le(1-c_0),
\end{align*}
have a unique solution.

\begin{proposition}\label{prop:div-curl}
	Let $\sigma \in [2,s]$ be an integer. Given $\vom, g\in H^{\sigma-1}(\Omega_f^+)$, $\vartheta\in H^{\sigma-\frac12}(\Gamma_f)$ with the compatiblity condition:
	\begin{align*}
	  -\int_{\Om_f^+} g dx=\int_{\Gamma_f} \vartheta ds,
	\end{align*}
	and $\vom$ satisfies
	\begin{align*}
		&\div\vom=0\quad \mathrm{in}\quad \Omega_f^+,\quad\int_{\Gamma^+}\om_3dx'=0,
	\end{align*}
	Then there exists a unique $\vu\in H^{\sigma}(\Omp)$ of the div-curl system (\ref{eq:div-curl}) so that
	\begin{align*}
		\|\vu\|_{H^{\sigma}(\Omega_f^+)}\le C\big(c_0,\|f\|_{H^{s+\f12}}\big)\Big(\|\vom\|_{H^{\sigma-1}(\Omega_f^+)}+\|g\|_{H^{\sigma-1}(\Omega_f^+)}
		 +\|\vartheta\|_{H^{\sigma-\frac12}(\Gamma_f)}+|\alpha_1|+|\alpha_2|\Big).
	\end{align*}
\end{proposition}

\subsection{Commutator estimate}

\begin{lemma}\label{lem:commutator}
	If $s>1+\frac{d}{2}$, then we have
	\begin{equation}
		\big\|[a, \langle\na\rangle^s]u\big\|_{L^2}\le C\|a\|_{H^{s}}\|u\|_{H^{s-1}}.
	\end{equation}
\end{lemma}

\subsection{Sobolev estimates of DN operator}
\begin{proposition}\label{prop:DN-Hs}
	If $f\in H^{s+\frac{1}{2}}(\bbT^2)$ for $s>\frac{5}{2}$, then it holds that for any $\sigma\in \big[-\frac{1}{2},s- \frac{1}{2}	 \big]$,
	\begin{equation}
		\|\mathcal{N}^\pm_f\psi\|_{H^{\sigma}}\le  K_{s+\frac{1}{2},f}\|\psi\|_{H^{\sigma+1}}.		
	\end{equation}
	Moreover, it holds that for any $\sigma\in \big[\frac{1}{2},s- \frac{1}{2}	\big]$,
	\begin{equation}
		\|\big(\mathcal{N}^+_f-\mathcal{N}^-_f\big)\psi\|_{H^\sigma}\le K_{s+\frac{1}{2},f}\|\psi\|_{H^\sigma},
	\end{equation}
	where $K_{s+\frac{1}{2},f}$ is a constant depending on $c_0$ and $||f||_{H^s}$.
\end{proposition}
\begin{proposition}\label{prop:DN-inverse}
	If $f\in H^{s+\frac{1}{2}}(\bbT^2)$ for $s>\frac{5}{2}$, then it holds that for any $\sigma\in \big[-\frac{1}{2},s- \frac{1}{2}\big]$,
	\begin{equation}
		\|\mathcal{G}^\pm_f\psi\|_{H^{\sigma+1}}\le  K_{s+\frac{1}{2},f}\|\psi\|_{H^{\sigma}},
	\end{equation}
	where $\mathcal{G}^\pm_f\triangleq\big(\mathcal{N}^\pm_f\big)^{-1}$.
\end{proposition}

%Let

\end{appendix}
\section*{Acknowledgment}
The authors wish to express their thanks to Prof. Zhifei Zhang and Prof. Wei Wang for suggesting the problem and for many helpful discussions. This work was supported by NSF of China under Grant No. 11871424.

\end{document}